\documentclass[11pt,a4paper]{article}
\pdfoutput=1

\usepackage[utf8]{inputenc}
\usepackage[T1]{fontenc}
\usepackage[english]{babel}
\usepackage{microtype}
\usepackage{amsmath}
\usepackage{amsfonts}
\usepackage{amssymb}
\usepackage{amsthm}
\usepackage{amscd}
\usepackage{latexsym}
\usepackage{thmtools}
\usepackage{mathtools} 
\usepackage{lmodern}
\usepackage{stmaryrd}
\usepackage{MnSymbol}
\usepackage{extarrows} 

\usepackage{hyperref}
\hypersetup{
    colorlinks=true,
    linkcolor=blue,
    urlcolor=RoyalBlue,
    citecolor=ForestGreen
    }
    
\usepackage[noabbrev,capitalize,nameinlink]{cleveref}
\crefname{prop}{Proposition}{Propositions}
\crefname{fact}{Fact}{Facts}

\usepackage[titles]{tocloft} 
\usepackage[shortlabels]{enumitem} 
\setlist[enumerate,1]{label={\upshape\alph*)}}
\setlist[enumerate,2]{label={\upshape\roman*)}}

\usepackage{xfrac}
\usepackage{graphicx}
\usepackage[dvipsnames]{xcolor} 
\usepackage{slashed}
\usepackage[autostyle=true]{csquotes}
\usepackage[all,cmtip]{xy}
\usepackage{tikz-cd}
\usepackage{cancel} 
\usepackage{xparse}

\usepackage[a4paper]{geometry}

\usepackage[citestyle=alphabetic,
            bibstyle=alphabetic,
            backend=bibtex,
            url=true,
            doi=false,
            isbn=false,
            giveninits=true]{biblatex}

\bibliography{absolute_cycles}
\AtEveryBibitem{%
  \ifentrytype{article}{
    \clearfield{url}%
    \clearfield{urldate}%
  }{}
  \ifentrytype{book}{
    \clearfield{url}%
    \clearfield{urldate}%
  }{}
  \ifentrytype{incollection}{
    \clearfield{url}%
    \clearfield{urldate}%
  }{}
} 

\declaretheorem[name=Theorem,numberwithin=section]{thm}
\declaretheorem[name=Lemma,numberlike=thm]{lem}
\declaretheorem[name=Corollary,numberlike=thm]{cor}
\declaretheorem[name=Proposition,numberlike=thm]{prop}

\declaretheorem[name=Definition,numberlike=thm,style=definition,qed=\(\lozenge\)]{defn}
\declaretheorem[name=Definition,numberlike=thm,style=definition,qed=\(\lozenge\),numbered=no]{defnintro}
\declaretheorem[name=Example,numberlike=thm,style=definition,qed=\(\lozenge\)]{ex}

\declaretheorem[name=Remark,numberlike=thm,style=definition,qed=\(\lozenge\)]{rem}

\newtheorem{thmintro}{Theorem}

\newcommand{\N}{\mathbb{N}}
\newcommand{\IN}{\mathbb{N}}

\newcommand{\IZ}{\mathbb{Z}}

\newcommand{\R}{\mathbb{R}}
\newcommand{\IR}{\mathbb{R}}
\newcommand{\C}{\mathbb{C}}

\newcommand{\cO}{\mathcal{O}}

\newcommand{\cF}{\mathcal{F}}

\newcommand{\cH}{\mathcal{H}}

\newcommand{\supp}{\mathrm{supp}}

\newcommand{\abs}{\mathrm{abs}}
\newcommand{\rel}{\mathrm{rel}}

\DeclareMathOperator{\im}{im}

\newcommand{\pt}{\mathrm{pt}} 
\newcommand{\loc}{\mathrm{loc}} 
\newcommand{\cc}{{cc}}

\newcommand{\olO}{\overline{\Omega}}
\newcommand{\dO}{\partial\Omega}
\newcommand{\dom}{\mathrm{dom}}

\newcommand{\fBH}{\mathfrak{B}(H)}

\newcommand{\lsim}{\lesssim}
\newcommand{\Dirac}{\slashed{D}}

\newcommand{\exc}{\mathrm{exc}}

\numberwithin{equation}{section} 

\author{Alexander Engel\thanks{Institut f{\"u}r Mathematik und Informatik der Universit{\"a}t Greifswald, Walther-Rathenau-Str.\,47, 17489 Greifswald, Germany, 
\href{mailto:alexander.engel@uni-greifswald.de}{alexander.engel@uni-greifswald.de}}
\and Matti Lyko\thanks{Institut f{\"u}r Mathematik und Informatik der Universit{\"a}t Greifswald, Walther-Rathenau-Str.\,47, 17489 Greifswald, Germany, 
\href{mailto:matti.lyko@uni-greifswald.de}{matti.lyko@uni-greifswald.de}}
}

\title{Regular boundary conditions and analytic $K$-homology}

\begin{document}

\maketitle

\begin{abstract}
  In their seminal work \cite{BDT1989}, Baum--Douglas--Taylor showed that under a certain compactness assumption, local boundary conditions for first-order elliptic operators over compact manifolds with boundary define cycles for relative $K$-homology. Using a systematic approach to regular boundary conditions, we generalize their result to elliptic operators of arbitrary order and to non-compact manifolds with potentially non-compact boundary. In addition, we show that under stronger regularity conditions, local boundary conditions even define cycles for absolute $K$-homology. In contrast to the relative classes, the resulting absolute $K$-homology classes are in general not independent of the choice of boundary condition.
\end{abstract}

\setcounter{tocdepth}{1}
\tableofcontents

\paragraph{Acknowledgements.}
We thank Rudolf Zeidler for encouriging us to write this paper based on the PhD thesis of the second named author. We acknowledge financial support by Deutsche Forschungsgemeinschaft DFG (grant EN 1163/7-1, project number 510244757, Uniformly elliptic boundary value problems).


\section{Introduction}

$K$-homology is the homology theory dual to $K$-theory. Its significance stems from the fact that elliptic (pseudo-)differential operators over manifolds naturally define classes in it, which allows one to extract the topological information stored in these operators. Using $K$-homology one can prove the Atiyah--Singer index theorem \cite{Baum2018,Baum2016}, derive localized analytic indices by exploiting the connection between $K$-homology and cyclic cohomology, and use this to prove the Novikov conjecture for hyperbolic groups \cite{CH1990}, and much more. It also plays a fundamental role in the Baum--Connes conjecture by constituting its domain \cite{BCH1994}. The idea for $K$-homology goes back to Atiyah \cite{Atiyah1969} and was refined by Kasparov \cite{Kasparov1975}.

In the case of a compact manifold $M$ with boundary $\partial M$, an elliptic operator $D$ over $M$ does not canonically define an (unbounded) self-adjoint operator on the corresponding $L^2$-space of sections on which $D$ acts, but one has to choose a boundary condition. In relation to $K$-homology, in their seminal paper \cite{BDT1989} Baum--Douglas--Taylor proved that local boundary conditions for an elliptic first-order operator $D$ which satisfy a certain regulairty property define cycles for the relative $K$-homology $K_*(M,\partial M)$, i.e.~a \emph{relative} cycle. It turns out that the corresponding relative $K$-homology class is actually independent of the chosen boundary condition. Moreover, they identify the image of this relative class under the $K$-homological boundary map $K_*(M,\partial M)\to K_{*-1}(\partial M)$ in geometrically interesting situations. In particular, they show (assuming $M$ to be spin) that the class of the Dirac operator over $M$ is mapped to the class of the corresponding Dirac operator over~$\partial M$.

In the present paper, we extend the result of Baum--Douglas--Taylor on the existence of relative cycles from boundary conditions simultaneously in the following directions:
\begin{itemize}
\item We generalize to non-compact manifolds with possibly non-compact boundary.
\item We generalize to higher-order differential operators.
\item We provide additional regularity assumptions under which boundary conditions define cycles in the $K$-homology group $K_*(M)$, i.e.~\emph{absolute} cycles.
\end{itemize}
We stress that the absolute cycle in the third point is over $M$ with its boundary, and not over the interior $M^\circ \coloneqq M \setminus \partial M$. That $D$ defines a class in $K_*(M^\circ)$ was established by Baum--Douglas--Taylor for first-order operators. Defining cycles for $K_*(M)$ is more subtle; the associated classes do depend on the boundary condition, while those in $K_*(M^\circ)$ do not. We discuss an explicit case of this in \cref{ex_euler_char}.

Further, we work with Lipschitz domains $\Omega$ in $M$ to allow for more general domains than manifolds with boundary (e.g., $\Omega$ could have corners). Conceptually, the Lipschitz assumption on the boundary $\dO$ (as opposed to arbitrarily rough boundaries) is inessential, but it streamlines the Sobolev space theory which forms our technical foundation.

Baum--Douglas--Taylor formulate their regularity condition needed for a relative cycle via compactness of resolvents. While this provides the most general statement, we found any generalizations to non-compact manifolds unsatisfying. Instead, we use a regularity notion based on inclusions of the domain into local Sobolev spaces. While Sobolev spaces are only one way to obtain compactness properties a la Baum--Douglas--Taylor, this approach is still quite general, and in line with current approaches to the general theory of boundary conditions for elliptic operators \cite{BaerBallmann2012,Grosse2014,Bandara2023,Baer2025}. Specifically, we work with the following notions:

\begin{defnintro}
Let $D$ be an elliptic differential operator of order $d$ over an open Lipschitz domain $\Omega \subset M$. A \emph{boundary condition} for $D$ is a closed extension $D_e$ of $D$ such that the Hilbert space adjoint $D_e^*$ in $L^2(\Omega)$ is a closed extension of the formal adjoint $D^\dagger$.
We call $D_e$ \emph{semi-regular} if $\dom(D_e)$ is in $H^d_\loc(\Omega)$, and \emph{regular} if both $D_e$ and $D_e^*$ are semi-regular. Moreover, we call $D_e$ \emph{local} if $\dom(D_e)$ is preserved by multiplication with any smooth, compactly supported function on $\olO$.
\end{defnintro}

The reader not accustomed to this approach to boundary conditions is referred to \cref{just_bc}, and to the above-cited literature on the theory of boundary conditions for elliptic operators, where this operator-theoretic definition is connected to the more familiar picture of specified data on the boundary. We also comment that our notion of locality is adopted from the first-order case as treated in Baum--Douglas--Taylor. For higher-order operators, it is somewhat restrictive, for example not every differential boundary condition need be local in this sense (see \cref{ExNeumann} below).

An abridged version of our main theorem reads as follows (we have neglected the grading assumptions, how the Fredholm modules are exactly defined and how they are related to each other):

\begin{thmintro}[{cf.~\cref{ExistencKHomClasses}}]\label{main_thm_intro}
Let $M$ be a smooth Riemannian manifold (possibly with boundary or incomplete) and $\Omega \subset M$ be an open Lipschitz domain. Denote by $\Omega^\circ \coloneqq \Omega\setminus \partial M$ and $\dO\coloneqq \olO\setminus\Omega^\circ$ the interior and boundary, respectively. Let $D$ be a symmetric differential operator over $\Omega$ which is elliptic over $\olO$, and let $D_e$ be a boundary condition for $D$.
Then:
\begin{itemize}
\item[i)] $D_e$ defines an absolute cycle for $K_*(\Omega^\circ)$.
\item[ii)] Assume that $D_e$ is local and semi-regular. Then $D_e$ defines a relative cycle for $K_*(\olO,\partial \Omega)$.
\item[iii)] Assume that $D_e$ is local, self-adjoint and regular. Then $D_e$ defines an absolute cycle for $K_*(\olO)$.
\end{itemize}
The corresponding classes in $K$-homology provided in i) and ii) do not depend on the boundary condition, but the class provided by iii) does in general.
\end{thmintro} 

We will also generalize \cref{main_thm_intro}.ii)\&iii) to non-local boundary conditions. Given a boundary condition $D_e$ for $D$, we consider the $C^*$-algebra generated by all smooth, compactly supported functions on $\olO$ preserving the domain of $D_e$. By Gelfand duality, it is of the form $C_0(X_e)$ for some space $X_e$. In fact, $X_e$ arises from $\olO$ by quotienting the boundary in a way prescribed by the boundary condition $D_e$. If $D_e$ is semi-regular, we show that it defines a relative cycle for $K_*(X_e,X_e\setminus\Omega^\circ)$. If it is actually self-adjoint and regular, we show that it defines an absolute cycle for $K_*(X_e)$. In particular, if the functions preserving $\dom(D_e)$ generate the $C^*$-algebra $C_0(\olO)$, then \cref{main_thm_intro} still holds in the given form for $D_e$. 

Index theory on manifolds with boundary is a topic of ongoing research. For example, in \cite{Chang2015} Chang--Weinberger--Yu define a relative Baum--Connes index map from relative $K$-homology to the $K$-theory of a relative group $C^*$-algebra. The image of the relative class of the Dirac operator under this map obstructs the existence of a positive-scalar-curvature metric with product structure near the boundary (see also \cite{Schick2021}), and Chang--Weinberger--Yu conjecture that the vanishing of this obstruction is also (stably) sufficient for the existence of such a metric.

While boundary conditions are not necessary to define relative classes for, say, the Dirac operator, they can nevertheless provide convenient representatives, enabling efficient computations. For example, in a forthcoming article the second named author gives an alternative proof of the delocalized APS-index theorem (see \cite{Piazza2014,Xie2014}) using boundary conditions in this way \cite{LykoAPS}. 

The absolute classes defined by local, self-adjoint, and regular boundary conditions differ from the relative classes in that they do depend on the boundary condition (see \cref{ex_khom_chiral}). Hence, \cref{main_thm_intro} enables the direct study of such boundary conditions via K-homology. While they are not the focus of this article, we sketch potential applications to spaces of boundary conditions, Agranovich--Dynin-type formulas, and positive scalar curvature in \cref{sec_questions}.

Let us say a few words about the proof of \cref{main_thm_intro}. To prove that a boundary condition defines a K-homology cycle, different (local) compactness statements have to be shown. In elliptic theory, such statements ultimately reduce to the Rellich--Kondrachov compactness theorem, stating that the inclusion of a positive-degree Sobolev spaces into $L^2$-space is locally compact. To achieve this reduction in the boundaryless case, one often appeals to pseudo-differential operators and their mapping properties on the Sobolev scale. Extending this argument to the situation with boundary provides significant difficulties, especially if the boundary is not assumed to be smooth. To replicate the style of argument used in the boundaryless case, we will therefore not work with the scale of Sobolev spaces directly. Instead, we will use a formal scale of spaces adapted to the given boundary condition, where the desired mapping properties are more easily established. Only afterwards will we appeal to (semi-)regularity to connect our formal scale to the Sobolev scale, and deduce compactness from Rellich--Kondrachov.

The use of this formal operator calculus as a convenient substitute for the pseudo-differential calculus and the Sobolev scale is well-established. Its use in index theory goes back at least to Connes--Moscovici \cite{Connes1995}, and is also used in the recent article by Fries \cite{Fries2025} mentioned below. For examples outside index theory, the formal operator calculus features in the theory of boundary conditions developed by Bär--Bandara \cite{Baer2025}, and the investigations of the Calderón projection by Booß-Bavnbeck--Lesch--Zhu \cite{BoossBavnbek2009}.

We are aware of two previous generalizations of the work of Baum--Douglas--Taylor to non-compact manifolds:
\begin{itemize}
\item Forsyth--Goffeng--Mesland--Rennie \cite{Forsyth2019}. This work sets up the theory of relative spectral triples, and describes how to obtain relative $K$-homology cycles from them. First-order elliptic operators are given as an example (mostly covering the minimal boundary condition), and the construction of absolute cycles is touched upon in Section 2.2.2 therein, but not systematically as we do it here.
\item Fries \cite{Fries2025}. A follow-up to the previously cited article, generalizing to higher-order operators. For elliptic operators, relative cycles given by the minimal boundary condition are covered.
\end{itemize}
Both papers center the non-commutative geometry of relative spectral triples, and aim at understanding their image under the K-homological boundary map. Hence, they mainly focus on relative cycles and the minimal boundary condition. In contrast, we aim for a systematic treatment relating boundary conditions and regularity to K-homology, also making absolute cycles relevant and naturally interesting for our purposes.

This paper was developed out of the PhD thesis of the second named author \cite{LykoThesis}, whose topic was to provide a version of the results of Baum--Douglas--Taylor in uniform $K$-homology for uniformly elliptic operators over (non-compact) manifolds of bounded geometry and (possibly non-compact) boundary (the boundaryless case was carried out earlier by the first named author \cite{Engel2014,Engel2019,Engel2018,Engel2018NYJM}). During the project we realized that there does not exist a complete generalization of the results of Baum--Douglas--Taylor to non-compact manifolds in the usual non-uniform setting. Based on the relevant part of the second named author's PhD thesis, this article attempts to provide that generalization.

\subsection*{Structure of this article}

We begin with a brief review of a few relevant results from the theory of Sobolev spaces in \cref{sec_sobolev}. In \cref{sec_bc} we then define and discuss boundary conditions alongside their properties (regularity in particular). \cref{SectionExamples} contains some examples of boundary conditions, as well as a comparison of the regularity notion used here with others in the literature. In \cref{SectionOpCalc} we present the formal operator calculus, and carry out the main technical computation needed for the proof of the main theorem. That proof is given in \cref{SecMainResult}, alongside a few other facts pertaining to $K$-homology classes defined by boundary conditions, and a generalization of the main result to non-local boundary conditions. We conclude the article with a some further questions in \cref{sec_questions}.

\section{Preliminaries on Sobolev spaces and elliptic operators}\label{sec_sobolev}

This section offers a review of local Sobolev spaces on domains, the Rellich--Kondrachov compactness theorem in \cref{prop_rellich_kondrachov}, and interior regularity for elliptic operators in \cref{IntReg}. These results form the foundation for our subsequent discussion of regular boundary conditions, whose domain is (by definition) made up of functions in the local Sobolev space.

Let $M$ be a smooth Riemannian manifold (possibly with boundary or incomplete) and let $\Omega \subset M$ be a \emph{Lipschitz domain}, i.e. an open subset of $M$ that is locally the area below the graph of a Lipschitz function (compare e.g. \cite[Def.~1.2.1.1]{Grisvard1985}). Denote by $\bar\Omega$ the closure of $\Omega$ in $M$, and $\Omega^\circ \coloneqq \Omega \setminus \partial M$. We use the volume density of the Riemannian metric on $M$ to define $L^2$-spaces for subsets like $\Omega$ of $M$.

\begin{defn}
For $s\geq 0$ we define $H^s_\loc(\Omega)$ as the space of all distributions $u$ on $\Omega$ such that for every Lipschitz domain $V \subset \Omega^\circ$ that is precompact in $\olO$, we have $u|_V \in H^s(V)$. We turn $H^s_\loc(\Omega)$ into a locally convex topological vector space by equipping it with the family of seminorms $\|-|_V\|_{H^s(V)}$ for every such $V$.

If we have a vector bundle $E$ over $\Omega$, equipped with a metric and connection, we analogously define $H^s_\loc(\Omega;E)$.
\end{defn}

\begin{rem}
Let us clarify a few aspects of the above definition.
\begin{itemize}
\item Recall that a distribution $u$ on $\Omega$ is an element of the dual space of the space of all smooth, compactly supported functions on $\Omega$. The restriction $u|_V$ is the distribution on $V$ given by $(u|_V)(f)=u(\tilde{f})$, where $\tilde{f}$ denotes the extension by $0$ of the smooth, compactly supported function $f$ on $V$.
\item The distributions in $H^s_\loc(\Omega)$ may touch the boundary of $\Omega$ in $M$ as well as the boundary of $M$ itself. 
\item To define $H^s(V)$ we exploit that the closure of $V$ in $\bar\Omega$ is compact to cover $V$ by finitely many smooth coordinate charts of $M$ to reduce to the case of bounded Lipschitz domains in Euclidean space. The seminorm $\|-|_V\|_{H^s(V)}$ depends on the choice of atlas for $V$, but not its equivalence class, and hence the structure of $H^s_\loc(\Omega)$ as a locally convex topological vector space is well-defined.
\item We only use the manifold $M$ as an external reference to tell when the open manifold $\Omega$ constitutes a Lipschitz domain and to have a notion of differentiability of arbitrary order on the closure $\olO$. Beyond that, $M$ is a largely superfluous datum. Indeed, if $M'$ is another manifold and $\Omega' \subseteq M$ another Lipschitz domain so that there exists a diffeomorphism between open neighborhoods of $\olO$ and $\olO'$, then $H^s_\loc(\Omega)$ and $H^s_\loc(\Omega')$ are isomorphic as locally convex topological vector spaces. The manifold $M$ will barely play a role in the rest of the paper.
\item A priori, there are multiple, possibly non-equivalent definitions of Sobolev spaces on open domains we could use. There is an intrinsic definition via functions on $V$ with finite Sobolev norm \cite[Definition~1.3.2.1]{Grisvard1985}, but also an extrinsic definition via restrictions of Sobolev functions on $\R^n$ to $V$ \cite[Definition~1.3.2.4]{Grisvard1985}. The former appears more natural, especially on manifolds, but the latter has favorable properties (embedding and compactness theorems, interpolation). On arbitrary domains the two may differ, but on Lipschitz domains they agree \cite[Section~1.4.3]{Grisvard1985}. We restrict ourselves to Lipschitz boundaries for $\Omega$ and the domains $V$ for this reason. Dropping this condition would cause no conceptual troubles, just technical ones.
\item Extending the previous point, we could also define $H^s_\loc(\Omega)$ using arbitrary pre-compact open domains $V$, using, say, the extrinsic definition for Sobolev spaces on arbitrary domains. Since every pre-compact domain is contained in a pre-compact Lipschitz domain (still assuming Lipschitz boundary for $\Omega$), and the Sobolev norm on the latter bounds that of the former, this would not change $H^s_\loc(\Omega)$ as a locally convex topological vector space.
\item Lastly, when working with Sobolev spaces of sections of a vector bundle, we shall often omit the bundle from the notation, and still write $H^s_\loc(\Omega)$. The implicit bundle should always be clear from context. \qedhere
\end{itemize}
\end{rem}

In the following lemma, we identify the interpolation spaces of Sobolev spaces on domains. For the definition of interpolation spaces, see \cite[Section~IV.2]{Taylor1996} or the original \cite{Calderon1964}. Interpolation spaces are a convenient tool to extend results for integer-degree Sobolev spaces to fractional degrees, see the proof of \cref{prop_rellich_kondrachov} below.

\begin{lem}\label{InterpolationSobolev}
Let $V\subset \Omega^\circ$ be an open Lipschitz domain which is precompact in $\bar\Omega$. Let $k\in\IN$ and $0\le \theta \le 1$. Then, the interpolation space between $L^2(V;E)$ and $H^k(V;E)$ is given by
$$
\left[L^2(V;E), H^k(V;E)\right]_\theta = H^{k\theta}(V;E) .
$$
\end{lem}
\begin{proof}
If $V$ is a smooth domain, this is proven in \cite[Thms.~9.6 \& 9.1 and (7.1)]{Lions1972}. But a close inspection of the proof in loc.cit. shows that it relies on the extension operator $H^s(U) \to H^s(\IR^n)$ if $U \subset \IR^n$ is an open and bounded subset of $\IR^n$. Since its existence is guaranteed if $\partial U$ is locally the graph of a Lipschitz function \cite[Thm.~1.4.3.1]{Grisvard1985}, we get $\left[L^2(V;E), H^k(V;E)\right]_\theta = H^{k\theta}(V;E)$ for all open Lipschitz domains of $\Omega^\circ$ which are precompact in $\bar\Omega$.
\end{proof}

We use this interpolation result to deduce the usual Rellich--Kondrachov compactness theorem for arbitrary Solbolev degree.

\begin{prop}\label{prop_rellich_kondrachov}
  Fix an open Lipschitz domain $V\subseteq \Omega^\circ$ that is precompact in $\olO$, and let $s>0$. Then, the embedding
  $$
    H^{s}(V;E) \longrightarrow L^2(V;E)
  $$
  is compact.
\end{prop}

\begin{proof}
  Suppose first that $s$ is an integer. If $V$ is in a coordinate chart over which $E$ is trivializable, the statement is covered by the classical Rellich--Kondrachov theorem, see e.g. \cite[Thm. 1.4.3.2]{Grisvard1985}. In general, due to its precompactness $V$ can be covered by a finite number of coordinate charts $U_i$ over which $E$ is trivializable. We may choose these charts such that their intersection with $V$ are Lipschitz domains. Then, using a partition of unity the compactness of $H^s(V;E)\hookrightarrow L^2(V;E)$ can be deduced from that of $H^s(V\cap U_i;E)\hookrightarrow L^2(V\cap U_i;E)$.

  The general case can be reduced to the case of integer $s$ by means of interpolation. Fix some integer $k\geq s$. By the \cref{InterpolationSobolev} we have $H^s(V;E)=[L^2(V;E),H^k(V;E)]_\theta$ for suitable $\theta$. It is a general fact that if $X,Y,Z$ are Banach spaces with $X$ being continuously included in $Y$, and $T\colon \, Y\to Z$ a continuous operator whose restriction to $X$ is compact, the operators $T_\theta\colon \, [Y,X]_\theta \to Z$ are compact for $0<\theta<1$, see \cite[Chapter V, Section~2]{LionsPeetre1964}. Taking $X=H^k(V;E)$, $Y=Z=L^2(V;E)$, and $T$ as the identity map on $L^2(V;E)$, the compactness of $H^s(V;E)\hookrightarrow L^2(V;E)$ thus follows from that of $H^k(V;E)\hookrightarrow L^2(V;E)$.
\end{proof}

We denote by $\Gamma_\cc(\Omega;E)$ the space of smooth sections which are compactly supported inside $\Omega^\circ$.

\begin{defn}
For a vector bundle $E$ over $\Omega$ we denote by $H^s_{0,\loc}(\Omega;E)$ the closure of $\Gamma_\cc(\Omega;E)$ inside $H^s_\loc(\Omega;E)$.
\end{defn}

Let $E$ and $F$ be vector bundles over $\Omega$, and $D\colon \, \Gamma_\cc(\Omega;E) \to \Gamma_\cc(\Omega;F)$ be a differential operator of order $d$. Through differentiation in the distributional sense, $D$ defines a continuous map\footnote{More concretely, it defines a bounded operator $H^s(V;E) \to H^{s-d}(V;F)$ for every open Lipschitz domain $V \subset \Omega^\circ$ that is precompact in $\bar\Omega$.\label{footnote_lem_differential_op_cont_sobolev}}
\begin{equation}\label{diff_sobolev}
  D \colon H^{s}_\loc(\Omega;E) \longrightarrow H^{s-d}_\loc(\Omega;F)
\end{equation}
for all $s\geq d$.

We are concerned with the case where $D$ is elliptic, i.e. its principal symbol $\sigma_D(\xi)$ is invertible for every non-zero cotangent vector $\xi\in T^*\Omega$. Let us recall interior regularity for elliptic operators in a form that will be sufficient for us.

\begin{prop}[{Interior regularity, \cite[Thm.\,V.11.1]{Taylor1996}}]\label{IntReg}
  Let $W$ be a smooth manifold without boundary, and $D$ be an elliptic operator of order $d$ acting on vector bundles $E, F$ over $W$. Let $V$ be open and precompact in $W$, $U$ be open and precompact in $V$, with $U$ and $V$ having smooth boundary. Then, for all distributions $u$ on $E$ with $Du \in L^2(W;F)$ it holds that $u \in H^d_\loc(W;E)$, and there is a constant $C$, dependent only on $U$ and $V$, such that
  $$
  \|u\|_{H^d(U;E)} \leq C \cdot \left( \|u\|_{L^2(V;E)} + \|Du\|_{L^2(V;F)} \right).
  $$
\end{prop}

Let the interior $\Omega^\circ$ take the role of $W$ in the preceeding proposition. If $D$ is only assumed to be elliptic over $\Omega^\circ$, the constant $C$ in \cref{IntReg} may blow up as $V$ and $U$ get arbitrarily close to the boundary of $\Omega$ or $M$. In our subsequent discussion of boundary conditions we shall need regularity up to the boundary, and to get it, we need to assume $D$ is elliptic up to the boundary in the following sense:

\begin{defn}
  A differential operator $D$ over $\Omega$ is said to be \emph{elliptic over $\olO$} if the vector bundles $E$ and $F$ together with all their structure extend to $\olO$, the coefficients of $D$ extend to $\olO$, and over each point $x\in\olO$ the principal symbol $\sigma_D(\xi)$ is invertible for all $\xi \in T^*_xM\setminus\{0\}$.  
\end{defn}

\begin{rem}\label{EllipticExt}
Let us explain how to get from our setting to the one of \cref{IntReg}. Suppose that $D$ is elliptic over $\olO$. First, attach a collar to the boundary of $M$ to embed $M$ in a boundaryless manifold $\tilde{M}$. Then, $\Omega^\circ$ is an open subset of $\tilde{M}$ with boundary $\dO^\circ = \dO \cup (\partial M \cap \Omega)$. Let $x\in\dO^\circ$. By assumption the coefficients of $D$ extend to $x$, hence they also extend to a neighborhood of $x$. If that neighborhood is chosen sufficiently small, the principal symbol of $D$ will still be invertible there. Taking the union of such neighborhoods we may produce an open neighborhood $W$ of $\olO$ in $\tilde{M}$ such that $D$ extends to an elliptic operator over $W$. The boundary of $W$ in $\tilde{M}$ can be chosen smooth.
\end{rem}

\section{Boundary conditions}\label{sec_bc}

This section is devoted to the prerequisites on boundary conditions, including the definition and properties of regular boundary conditions as we will use them in this article. This section is mainly concerned with the general theory, examples will be given in \cref{SectionExamples} below.

Let $\Omega$, $E$, $F$, and $D$ be as in the previous section. We will consider $D$ as a densely defined, unbounded operator from $\Gamma_\cc(\Omega;E) \subset L^2(\Omega;E)$ to $L^2(\Omega;F)$. The formal adjoint of $D$, denoted by $D^\dagger$, analogously defines an unbounded operator from $\Gamma_\cc(\Omega;F) \subset L^2(\Omega;F)$ to $L^2(\Omega;E)$.

\begin{defn}
\label{DefBdryCond}
A \emph{boundary condition} for $D$ is a closed extension $D_e$ of $D$ such that the Hilbert space adjoint $D_e^*$ is a closed extension of the formal adjoint $D^\dagger$.
\end{defn}

It may not seem intuitive that this operator-theoretic definition captures the concept of a boundary condition. \cref{MinMaxDomain} and \cref{just_bc} below should assure the reader that such a closed extension $D_e$ of $D$ does indeed only add information on the boundary. On a deeper level, one can often associate a Sobolev-like space of boundary data to the operator $D$, so that boundary conditions in our sense are canonically identified with closed subspaces of this boundary data space, recovering the intuitive picture of a boundary condition. This is carried out for first-order operators and compact boundaries in \cite{BaerBallmann2012}, higher-order operators and compact boundaries in \cite{Bandara2023}, and first-order operators and non-compact boundaries (under mild boundedness hypotheses) in \cite{Baer2025}.

\subsection{Minimal and maximal boundary condition}

While \cref{SectionExamples} below is devoted to examples of boundary conditions, the following two examples are needed for the general theory.

\begin{defn}
The \emph{minimal boundary condition} $D_{\min}$ of $D$ is the closure of $D$ as an unbounded operator from $L^2(\Omega;E)$ to $L^2(\Omega;F)$, and the \emph{maximal boundary condition} $D_{\max}$ is the (Hilbert space) adjoint $D_{\max} \coloneqq (D^\dagger)^*$ of $D^\dagger$ as an unbounded operator from $L^2(\Omega;F)$ to $L^2(\Omega;E)$.
\end{defn}

\begin{rem}\label{rem_min_max_conditions}
The minimal and maximal extensions can be characterized as follows:
\begin{enumerate}
\item A section $u\in L^2(\Omega;E)$ is in the domain of $D_{\min}$ if and only if there is a sequence $(u_n)_{n\in\N} \subset \Gamma_{cc}(\Omega;E)$, with $\|u-u_n\|_{L^2} \to 0$ and $\sup_n \|Du_n\|_{L^2}< \infty$ \cite[Lem.~1.8.1]{HigsonRoe2000}.

Equivalently, $\dom(D_{\min})$ is the closure of $\Gamma_{cc}(\Omega;E)$ in the graph norm
\[
\|u\|_D=\|u\|_{L^2}+ \|Du\|_{L^2},
\]
and $D_{\min}$ is the unique extension of $D$ to this closure.
\item A section $u\in L^2(\Omega;E)$ is in the domain of $D_{\max}$ if and only if there is $v\in L^2(\Omega;F)$ such that for all $w\in \Gamma_{cc}(\Omega;F)$ it holds that
\begin{equation}\label{eq_defn_Dmax}
\langle u , D^\dagger w\rangle_E = \langle v, w \rangle_F .
\end{equation}
Then $D_{\max}u= v$. 
\item\label{item_min_max_boundaryconditions} Both $D_{\min}$ and $D_{\max}$ are closed extensions of $D$, and it holds that $D_{\min} \subseteq D_{\max}$. Moreover, we have $D_{\max}=(D^\dagger)^*=(D_{\min}^\dagger)^*$, and hence $D_{\min}=(D_{\max}^\dagger)^*$. Thus, minimal and maximal boundary condition are indeed boundary conditions in the sense of \cref{DefBdryCond}. 
\item\label{BCsandwich} Any boundary condition $D_e$ for $D$ satisfies $D_{\min}\subseteq D_e\subseteq D_{\max}$. In fact, a closed extension satisfies \cref{DefBdryCond} if and only if this condition holds.
\item\label{differential_op_cont_sobolev} By definition, $D_{\max}$ acts in the distributional sense, which is also how $D$ acts on the local Sobolev spaces $H^s_\loc(\Omega;E)$ (see \eqref{diff_sobolev}). Hence, we have
$$ 
  D_{\max} = D \text{ as maps } \dom(D_{\max}) \cap H^{s}_\loc(\Omega;E)  \longrightarrow H^{s-d}_\loc(\Omega;F)
$$
for all $s\geq d$. Since any boundary condition is contained in the maximal one, the same is true with $D_{\max}$ replaced by an arbitrary boundary condition $D_e$.
\qedhere
\end{enumerate}
\end{rem}

In the following proposition, we will use interior regularity to characterize the minimal boundary condition. We note that it requires our assumption that $D$ be elliptic over $\olO$ instead of just $\Omega$, as it may fail otherwise.

\begin{prop}\label{MinDomSobolev}
  Suppose that $D$ is elliptic over $\olO$. Let $U$ be an open and precompact subset of $M$ whose boundary is locally a Lipschitz graph.
\begin{enumerate} 
\item 
  For $u \in \dom(D_{\min})$ we have $u \in H^d_{0,\loc}(\Omega;E)$ and an estimate
\begin{equation}\label{eq_estimate_Sobolev_graph}
\|u\|_{H^d(U\cap\Omega^\circ;E)} \leq C\|u\|_D
\end{equation}
with a constant $C$ depending only on $U$.
Hence, we have a continuous inclusion
\begin{equation}\label{eq_MinDomEmbedding}
\dom(D_{\min}) \hookrightarrow H^d_{0,\loc}(\Omega;E),
\end{equation}
where the domain is equipped with the graph norm.
\item\label{item_MinDomSobolev_ii}  
  If $u\in L^2(\Omega;E)$ is supported inside $U$,\footnote{Note that the support of $u$ may touch the boundary of $\Omega$ or of $M$.} then
  $$
  u \in \dom(D_{\min}) \quad \Leftrightarrow \quad u \in H^d_{0,\loc}(\Omega;E) ,
  $$
  and there is a constant $C$, dependent only on $U$, such that
\begin{equation}\label{eq_estimate_Sobolev_graph_both}
  C^{-1} \| u \|_D \leq \|u\|_{H^d(U\cap\Omega^\circ;E)} \leq C\|u\|_D .
\end{equation}
\end{enumerate}
\end{prop}

\begin{proof}
  To prove \eqref{eq_estimate_Sobolev_graph}, we first suppose $u \in \Gamma_\cc(\Omega;E)$. By \cref{EllipticExt} there is an open manifold $W$ containing $\bar\Omega$ such that $D$ extends to an elliptic operator over $W$. Let $U^\prime$ be any open neighbourhood of $\bar U$ that is precompact in $W$, $V'$ a neighbourhood of $U'$ that is also precompact in $W$, and let them have smooth boundary. \cref{IntReg} gives 
  $$
  \|u\|_{H^d(U';E)} \leq C\cdot\left( \|u\|_{L^2(V';E)} + \|Du\|_{L^2(V';F)} \right)
  $$
  for a constant $C$ depending only on the choices for $U'$ and $V'$. Since $U \cap \Omega^\circ \subset U'$, we get $\|u\|_{H^d(U \cap \Omega^\circ;E)} \le \|u\|_{H^d(U';E)}$. Since $u$ and $Du$ are still supported in $\Omega$, we further have $\|u\|_{L^2(V';E)} + \|Du\|_{L^2(V';F)} = \|u\|_{L^2(V'\cap \Omega;E)} + \|Du\|_{L^2(V' \cap \Omega;F)}$, and the right hand side can be estimated from above by the graph norm of $u$ just because $V'\cap \Omega \subset \Omega$. Everything together proves \eqref{eq_estimate_Sobolev_graph} for $u \in \Gamma_\cc(\Omega;E)$.
  
  The estimate \eqref{eq_estimate_Sobolev_graph} holds for every $u \in \Gamma_\cc(\Omega;E)$ and every precompact Lipschitz domain $U$ of $M$, and hence for any seminorm defining $H^d_\loc(\Omega;E)$. By definition, $\Gamma_\cc(\Omega;E)$ is dense in $\dom(D_{\min})$ with the graph norm, i.e. a general $u \in \dom(D_{\min})$ is a limit of a sequence $(u_n) \subset \Gamma_\cc(\Omega;E)$ in graph norm. From \eqref{eq_estimate_Sobolev_graph} we conclude that the sequence also converges to $u$ in every seminorm defining $H^d_\loc(\Omega;E)$, and since each $u_n$ is in $\Gamma_\cc(\Omega;E)$ this shows $u \in H^d_{0,\loc}(\Omega;E)$ and it still satisfies \eqref{eq_estimate_Sobolev_graph}.
  
To prove \eqref{eq_estimate_Sobolev_graph_both} first suppose again that $u \in \Gamma_{cc}(\Omega;E)$ and that it has support in~$U$. Because $\supp(u) \subseteq \Omega^\circ$ by definition, $u$ is automatically supported in $V \coloneqq U \cap \Omega^\circ$. We already know \eqref{eq_estimate_Sobolev_graph}, because $u \in \dom(D_{\min})$ in this case. Since the coefficients of $D$ are bounded over the precompact $V$, we also have the estimate\footnote{Here and in the following, the notion $a\lsim b$ shall be used to signify $a\leq Cb$ for some constant $C$ that is universal relative to the present context. For example, in the following estimate, the implicit constant may depend on $V$ and $D$, but not the section $u$.}
  $$
  \|u\|_D = \|u\|_{L^2(V;E)} + \|Du\|_{L^2(V;F)} \lsim \|u\|_{H^d(V)} ,
  $$
  proving the claimed norm equivalence \eqref{eq_estimate_Sobolev_graph_both} for $u \in \Gamma_{cc}(\Omega;E)$ supported in $U$.

  If $u\in L^2(\Omega;E)$ is in $\dom(D_{\min})$ or in $H^d_{0,\loc}(\Omega;E)$ and has support in $U$, then it can be approximated in the corresponding norm by elements of $\Gamma_{cc}(\Omega;E)$ that also have support in $U$. The norm equivalence for such sections shows that the approximating sequence also converges in the other norm. We conclude that $u\in\dom(D_{\min})$ if and only if $u\in H^d_{0,\loc}(\Omega;E)$, and the norm equivalence for $u$ follows from that for the approximating sequence.
\end{proof}

The next lemma establishes that minimal and maximal boundary condition coincide on the interior.

\begin{prop}\label{MinMaxDomain}
  Let $D$ be elliptic over $\olO$ and $u \in L^2(\Omega;E)$ have compact support in~$\Omega^\circ$. Then it holds that:
$$
u \in \dom(D_{\min}) \quad \Leftrightarrow \quad u\in \dom(D_{\max}).
$$
\end{prop}
\begin{proof}
Assume $u \in \dom(D_{\max})$, i.e. we have in the distributional sense $D_{\max} u \in L^2(\Omega;F)$ (see \cref{rem_min_max_conditions}.\ref{differential_op_cont_sobolev}). From interior regularity (\cref{IntReg}) and \cref{EllipticExt} we conclude $u \in H^d_\loc(\Omega;E)$. In fact, since the support of $D_{\max} u$ is contained in the support of $u$, we have $u \in H^d(V;E)$ for an open Lipschitz neighbourhood $V \subset \Omega^\circ$ of the support of $u$ and such that $V$ is precompact in~$\Omega^\circ$ (not just in $\bar\Omega$). By the latter, $u$ is in $H^d_{0,\loc}(\Omega;E)$ and so by \cref{MinDomSobolev}.\ref{item_MinDomSobolev_ii} we have $u \in \dom(D_{\min})$.
\end{proof}

\begin{rem}\label{just_bc}
  As noted in \cref{rem_min_max_conditions}.\ref{BCsandwich}, boundary conditions satisfy $D_{\min}\subseteq D_e \subseteq D_{\max}$. Hence, \cref{MinMaxDomain} implies that $u\in L^2(\Omega;E)$ with compact support inside the interior $\Omega^\circ$ is in $\dom(D_e)$ if and only if it is in $\dom(D_{\min})$, which is the case if and only if it is in $\dom(D_{\max})$. In other words, sections supported in the interior do not feel the choice of boundary condition, as indeed it should be.
\end{rem}

The next two lemmas concern continuity of multiplication operators.

\begin{lem}\label{MultMinDomain}
  Let $D$ be elliptic over $\olO$, and let $f\in C^\infty_{c}(\Omega)$ be a smooth and compactly supported function. Then, 
  $$
  \dom(D_{\min}) \longrightarrow \dom(D_{\min}) , \quad u \longmapsto fu
  $$
  is a well-defined, continuous map.
\end{lem}

\begin{proof}
  Let $U$ be a precompact, open neighborhood of $\supp(f)$ in $M$ with a Lipschitz boundary. The composition
  $$
  \dom(D_{\min}) \hookrightarrow H^d_{0,\loc}(\Omega;E) \xlongrightarrow{f} H^d_{0,\loc}(\Omega;E) \cap H^d(U\cap\Omega^\circ;E) \longrightarrow \dom(D_{\min})
  $$
  is well-defined and continuous. This is due to \eqref{eq_MinDomEmbedding} and \eqref{eq_estimate_Sobolev_graph_both} of \cref{MinDomSobolev} for the first and third arrow respectively, and for the second it is due to the fact that multiplication with smooth, compactly supported functions is continuous on $H^d_{0,\loc}(\Omega;E)$. 
\end{proof}

\begin{lem}\label{MultInt}
Suppose $D$ is elliptic over $\bar\Omega$ and let $f\in C^\infty_{cc}(\Omega)$. Then, multiplication by $f$ is a bounded map
$$
\dom(D_{\max}) \longrightarrow \dom(D_{\min}).
$$
\end{lem}

We note that by \cref{rem_min_max_conditions}.\ref{BCsandwich}, $D_{\max}$ could be replaced by an arbitrary boundary condition in \cref{MultInt}.

\begin{proof}
Fix some open neighborhood $U$ of $\partial\Omega \coloneqq \bar\Omega \setminus \Omega^\circ$ that does not intersect $\supp(f)$, and let $P \colon \, \dom(D_{\max})\to \dom(D_{\max})$ be the orthogonal projection onto the closed subspace of sections with support outside $U$.\footnote{Here we consider $\dom(D_{\max})$ as a Hilbert space under the inner product $\langle \cdot,\cdot\rangle_D=\langle\cdot,\cdot\rangle_{L^2} + \langle \cdot , |D_{\max}|^2\cdot \rangle_{L^2}$, which induces a norm equivalent to $\|\cdot \|_D$.} Notice that if $u\in \dom(D_{\max})$ lies in the kernel of $P$, then its support is contained in $U$. By \cref{MinMaxDomain} $\im(P)$ is contained in $\dom(D_{\min})$, so that $u\mapsto fu, \, \im(P) \to \dom(D_{\min})$ is bounded by \cref{MultMinDomain}. On the other hand, if $u\in \dom(D_{\max})$ is in the kernel of $P$, then its support is contained in $U$, whence $fu=0$. Hence $u\mapsto fu$ factors as 
$$
\dom(D_{\max})\xlongrightarrow{P} \im(P) \xlongrightarrow{f} \dom(D_{\min}),
$$ 
which is a composition of bounded maps.
\end{proof}

\subsection{Regularity, locality, and gradings}

Now, we introduce the notion of (semi-)regular boundary conditions that we will use in this paper. 

\begin{defn}\label{DefReg}
The boundary condition $D_e$ is \emph{semi-regular}, if $\dom(D_e)$ equipped with the graph norm is continuously included in $H^d_\loc(\Omega;E)$. It is \emph{regular} if both $D_e$ and $D_e^*$ are semi-regular.
\end{defn}

\cref{MinDomSobolev} implies that the minimal boundary condition is semi-regular. However, the maximal boundary condition is generally not semi-regular (see \cref{ExMinMax}), so that $D_{\min}$ is generally not regular. Further examples are defered to \cref{SectionExamples}.

\begin{rem}\label{SemiRegLocalNormEq}
If $D_e$ is semi-regular, then every $u\in\dom(D_e)$ satisfies an estimate of the form $\|u\|_{H^d(V;E)} \lsim \|u\|_D$ for every open Lipschitz domain $V\subseteq \Omega^\circ$ that is precompact in $\olO$. 
By \cref{differential_op_cont_sobolev} of \cref{rem_min_max_conditions} (with \cref{footnote_lem_differential_op_cont_sobolev}) there is also the converse estimate
\begin{equation}\label{eq_semireg_equiv}
\| u\|_{L^2(V;E)} + \|D_e u\|_{L^2(V;F)} \lsim \|u\|_{H^d(V;E)} .
\end{equation}
In particular, the graph norm and the $H^d(V;E)$-norm are equivalent for all $u\in\dom(D_e)$ with $\supp(u)\subseteq \overline{V}$. In fact, regularity estimates are a good way to establish that boundary conditions are semi-regular.
\end{rem}

\begin{lem}\label{MultSemiReg}
Let $D_e$ be a semi-regular boundary condition, and let $f\in C^\infty_c(\olO)$ preserve the domain of $D_e$, i.e. $f\cdot\dom(D_e)\subseteq \dom(D_e)$. Then,
$$
\dom(D_e) \longrightarrow\dom(D_e), \quad u\longmapsto fu
$$ 
is a bounded map (using the graph norm on domain and target).
\end{lem}

\begin{proof}
Let $V \subset \Omega^\circ$ be an open Lipschitz domain that is precompact in $\bar\Omega$ and such that $\supp(f) \subset \bar V$.
We consider the composition
\[
  \dom(D_e) \hookrightarrow H^d_{\loc}(\Omega;E) \cap \dom(D_e) \xlongrightarrow{f} H^d_{\loc}(\Omega;E) \cap \dom(D_e) \cap L^2(\bar V; E) \to \dom(D_e),
\]
where both $H^d_{\loc}(\Omega;E) \cap \dom(D_e)$ and $H^d_{\loc}(\Omega;E) \cap \dom(D_e) \cap L^2(\bar V; E)$ are equipped with the subspace topology induced from $H^d_{\loc}(\Omega;E)$. The first arrow is continuous by assumption ($D_e$ is semi-regular), the second arrow is continuous since multiplication by smooth, compactly supported functions is continuous on local Sobolev spaces, and the third arrow is continuous due to \eqref{eq_semireg_equiv}.
\end{proof}

Recall that a \emph{grading} on the vector bundle $E$ (equipped with a metric and connection) is a unitary bundle morphism with $\kappa^2=I$. A $p$-multigrading on $E$, $p\in\N_{>0}$, consists of $p$ odd unitary bundle morphisms $\varepsilon_1,\dots,\varepsilon_p$ such that $\varepsilon_i\varepsilon_j + \varepsilon_j\varepsilon_i = -2\delta_{ij}$. Note that the grading operator $\kappa$ and the mutigrading operators $\varepsilon_j$ extend to unitary (multi-)grading operators on $L^2(\olO;E)$. The differential operator $D$ is then called \emph{odd} if $D\kappa=-\kappa D$, and \emph{multigraded} if $\varepsilon_jD=D\varepsilon_j$. Conventionally, graded vector bundles are also refered to as \emph{$0$-multigraded} bundles.

\begin{defn}
Let $D\colon \Gamma_{cc}(\Omega;E)\to \Gamma_{cc}(\Omega;E)$ be a differential operator acting on sections of a $p$-multigraded vector bundle, $p\geq 0$. Suppose that $D$ is odd and $p$-multigraded. A boundary condition $D_e$ is called \emph{odd} if the grading operator $\kappa$ of $E$ preserves $\dom(D_e)$, and $D_e \kappa = - \kappa D_e$. The boundary condition is called \emph{$p$-multigraded} if every multigrading operator $\varepsilon_j$ of $E$ preserves $\dom(D_e)$ and $D_e\varepsilon_j=\varepsilon_j D_e$. 
\end{defn}

The following lemma is deduced rather directly from the definition of the adjoint boundary condition:

\begin{lem}\label{AdjointMultigraded}
A boundary condition is odd resp.\,multigraded if and only if its adjoint is odd resp.\,multigraded.
\end{lem}

The other major assumption needed for our main result is that of locality, adopted from \cite{BDT1989}, which is expressed through compatibility with multiplication by smooth functions.

\begin{defn}\label{def_local}
A boundary condition $D_e$ for $D$ is called \emph{local} if $f\cdot\dom(D_e)\subseteq \dom(D_e)$ for all $f\in C^\infty_c(\olO)$. 
\end{defn}

We note that for higher-order operators, this notion of locality is quite restrictive. For example, differential boundary conditions of strictly positive order will generally not be local in this sense, see \cref{ExNeumann} below.

\subsection{Local boundary conditions for first-order operators}

This subsection is largely technical in nature, and we need its results primarily to extend regularity and adjointess results for certain examples from the compact to the non-compact setting (\cref{RegAbsRelChiral} below). If those examples are not of particular interest to the reader, skipping this section is encouraged.

We will assume that $D$ is first-order, and $D_e$ a local boundary condition. A helpful feature of the first-order case is that $[f,D]=-i\sigma_D(df)$ for $f\in C^\infty_c(\olO)$, a priori unbounded operators with domain $\Gamma_{cc}(\Omega;E)$, extend to a bounded operators on $L^2(\olO;E)$. Hence, the commutator $[f,D_e]$, which we can define because $D_e$ is local, also extends to a bounded operator.

\begin{lem}\label{LocalFirstOrder}
Let $D_e$ be a boundary condition for a first-order differential operator~$D$. Then, $D_e$ is local if and only if $D_e^*$ is local.
\end{lem}

\begin{proof}
The roles of $D_e$ and $D_e^*$ are symmetric, so it suffices to prove that if $D_e$ is local, so is $D_e^*$. It further suffices to prove that any real-valued function $f\in C^\infty_c(\olO)$ preserves $\dom(D^*_e)$. Let $u\in \dom(D_e^*)$. We need to produce a vector $v\in L^2(\Omega)$ such that for every $w\in\dom(D_e)$ it holds that $\langle v,w\rangle = \langle fu,Dw\rangle$. Write 
$$
\langle fu,Dw\rangle = \langle u,fDw\rangle = \langle u, D(fw)\rangle + \langle u,[f,D]w\rangle =\langle fD_e^*u, w\rangle + \langle u,[f,D]w\rangle  \; .
$$
Since $D$ is first-order and $D_e$ local, $[f,D_e]$ is a bounded operator on $L^2(\Omega)$. Thus, $v \coloneqq fD_e^*u + [f,D_e]^*u$ is a well-defined element of $L^2(\Omega)$, for which it holds that
$$
\langle fu,Dw\rangle =\langle fD_e^*u, w\rangle + \langle u,[f,D]w\rangle = \langle v,w\rangle .
$$
This concludes the proof.
\end{proof}

\begin{lem}\label{ContMultFirstOrder}
  Let $D_e$ be a local boundary condition for a first-order differential operator~$D$, and let $\psi\in C^\infty_c(\olO)$. Then, 
  $$
    \dom(D_e) \longrightarrow \dom(D_e) \; , \; u \longmapsto \psi u
  $$
  is continuous.
\end{lem}

\begin{proof}
  Let $u\in \dom(D_e)$. It needs to be shown that $\|\psi u\|_D\lsim \|u\|_D$. Write
  $$
    D(\psi u) = \psi Du + [D,\psi] u .
  $$
  The commutator $[D_e,\psi]$ is a bounded operator on $L^2(\olO;E)$, hence
  \begin{align*}
    \|\psi u\|_D &= \| \psi u\|_{L^2(\olO;E)} + \| D(\psi u)\|_{L^2(\olO;F)} \\
    &= \| \psi u\|_{L^2(\olO;E)} + \| \psi Du \|_{L^2(\olO;F)} + \|\sigma_D(d\psi) u\|_{L^2(\olO;F)} \\
    &\lsim \|u\|_{L^2(\olO;E)} + \| Du\|_{L^2(\olO;F)} \\
    &= \|u\|_D .
  \end{align*}
  This concludes the proof.
\end{proof}

We will conclude this section with a series of lemmas stating that for local boundary conditions for first-order operators properties like (semi-)regularity can be verified in coordinate charts. 

\begin{lem}\label{RegLocalized}
  Let $D$ be a first-order elliptic differential operator over $\olO$, and $D_e$ a local boundary condition for $D$. Then, the following are equivalent:
  \begin{enumerate}
    \item\label{cond_reg} $D_e$ is semi-regular. 
    \item\label{cond_local} For every Lipschitz domain $V\subseteq \Omega^\circ$ that is precompact in $\olO$, every $\psi\in C_{c}^\infty(\olO)$ with support in $\overline{V}$, and all $u\in\dom(D_e)$ it holds that $\psi u \in H^1(V;E)$ with estimates
    $$
      \|\psi u \|_D \lsim \|\psi u \|_{H^1(V;E)} \lsim \|\psi u \|_D .
    $$
    \item\label{cond_coords} The condition \ref{cond_local} holds for every $V$ that is a coordinate (half-)ball.
  \end{enumerate}
  Moreover, $D_e$ is regular if and only if condition \ref{cond_local} or \ref{cond_coords} holds for both $D_e$ and $D_e^*$.
\end{lem}

\begin{proof}
  The implication \ref{cond_reg}$\Rightarrow$\ref{cond_local} is due to \cref{MultSemiReg} and \cref{SemiRegLocalNormEq}. Since \ref{cond_coords} is a special case of \ref{cond_local}, we have to show \ref{cond_coords}$\Rightarrow$\ref{cond_reg}. To that end, fix a locally finite open cover $\{V_i\}$ of $\olO$ by coordinate (half-)balls that are precompact in $\olO$, and let $\{\psi_i\}$ be a subordinate partition of unity. 
  
  It needs to be shown that for every Lipschitz domain $V$ that is precompact in $\olO$ and every $u\in\dom(D_e)$, it holds that $u|_V \in H^1(V;E)$ and that $\|u\|_{H^1(V;E)}\lsim \|u\|_D$. Since $V$ is precompact, it is covered by a finite number of the $V_i$. By assumption \ref{cond_coords} it holds that $\psi_i u \in H^1(V_i;E)$ for all $i$. Hence, 
  $$
    u|_V = \sum_{V_i\cap V\neq\emptyset} \psi_i u
  $$
  is a finite sum of Sobolev functions, hence in $H^1(V;E)$. Moreover,
  $$
    \|u\|_{H^1(V;E)} \leq \sum_{V_i\cap V\neq\emptyset} \|\psi_i u\|_{H^1(V;E)} \lsim \sum_{V_i\cap V\neq\emptyset} \|\psi_i u\|_{D} \lsim \|u\|_D ,
  $$
  where we used \ref{cond_coords} and \cref{ContMultFirstOrder}. 
  
  Lastly, if $D_e$ is local, then so is $D_e^*$ by \cref{LocalFirstOrder}. Hence, the equivalence of \ref{cond_reg}, \ref{cond_local}, and \ref{cond_coords} holds for $D_e^*$ as well. This gives the claimed equivalence for regularity, concluding the proof.
\end{proof}

\begin{lem}\label{LocalizeAdjoint}
  Let $D$ be a symmetric first-order elliptic differential operator over $\olO$, and let $D_e$ and $D_d$ be local boundary conditions for $D$. Assume that
  $$
    \left\{ u\in\dom(D^*_e) \, \colon \, \supp(u)\subseteq\olO \text{ is compact} \right\}
  $$
  is dense in $\dom(D_e^*)$, and the same holds for $D_d$ instead of $D_e^*$. Then, the following are equivalent:
  \begin{enumerate}
    \item\label{adj} $D_e^*=D_d$.
    \item\label{adj_local} For every Lipschitz domain $V\subseteq \Omega^\circ$ that is precompact in $\olO$, compact $K \subset \overline{V}$ that does not intersect some neighborhood of $\partial V\setminus\dO$, and every $u\in L^2(\olO;E)$ with support in $K$ it holds that $u\in\dom(D_d)$ if and only if
    \begin{equation}
      \exists w\in L^2(V;E) \, \forall v \in \dom(D_e) \cap L^2(V;E) \colon \langle u, D_ev \rangle_{L^2} = \langle w, v \rangle_{L^2} . \tag{\textasteriskcentered}\label{loc_adj_cond}
    \end{equation}
    \item\label{adj_coords} The condition \ref{adj_local} holds for every $V$ that is a coordinate (half-)ball.
  \end{enumerate}
\end{lem}

\begin{proof}
  We begin with \ref{adj}$\Rightarrow$\ref{adj_local}. Recall that $u\in\dom(D_e^*)$ if and only if
  \begin{equation}
      \exists w\in L^2(\Omega;E) \, \forall v \in \dom(D_e) \colon \langle u, D_ev \rangle_{L^2} = \langle w, v \rangle_{L^2} . \tag{\textasteriskcentered\textasteriskcentered}\label{dom_adj}
    \end{equation}
  Let $V,K,u$ be as stated. If $u\in\dom(D_d)$, then \eqref{loc_adj_cond} is a consequence of $D_d$ being the adjoint of $D_d$, since $w=D_du$ is automatically supported in $K$ if $u$ is. Conversely, if $u$ satisfies \eqref{loc_adj_cond}, then it is in $\dom(D_e^*)$. This is almost by definition, except that it needs to be argued that \eqref{dom_adj} needs to be verified only for $v\in \dom(D_e) \cap L^2(V;E)$. To see this, pick a bump function $\psi \in C^\infty_c(\olO)$ with support in $\overline{V}$ and $\psi|_K=1$. Such a function exists due to our assumption that $K$ does not intersect some neighborhood of $\partial V\setminus\dO$. Then, for an arbitrary $v\in\dom(D_e)$ it holds that $\langle w, v \rangle_{L^2}=\langle w, \psi v \rangle_{L^2}$, and $\psi v \in\dom(D_e)$ due to the assumed locality. This proves \ref{adj}$\Rightarrow$\ref{adj_local}.

  Since \ref{adj_coords} is a special case of \ref{adj_local}, it remains to show \ref{adj_coords}$\Rightarrow$\ref{adj}. To that end fix an open cover of $\olO$ by pre-compact coordinate (half-)balls $\{V_i\}$, and let $\{\psi_i\}$ be a subordinate partition of unity such that the support of $\psi_i$ does not intersect some neighborhood of $\partial V_i\setminus\dO$. Also fix functions $\varphi_i\in C^\infty_{c}(V_i)$ satisfying $\varphi_i\psi_i = \psi_i$. We will first show that $\dom(D_d)\subseteq\dom(D_e^*)$. To that end, fix $u\in\dom(D_d)$. By our assumption on $D_d$, we may assume $u$ to have compact support. Because $D_d$ is local, $\psi_i u$ is in $\dom(D_d)$ as well. By \ref{adj_coords} there is $w_i \in L^2(V_i;E)$ such that 
  $$
    \langle \psi_i u , D_e v \rangle_{L^2} = \langle w_i, v \rangle_{L^2} \quad \forall v \in \dom(D_e) \cap L^2(V_i;E) .
  $$
  Since $u$ is compactly supported, all but a finite number of $w_i$ are zero, hence $w \coloneqq \sum_i w_i$ is a well-defined element of $L^2(\Omega;E)$. We will show that this $w$ satisfies the property \eqref{dom_adj}. To that end, calculate that
  \begin{align*}
    \langle w, v\rangle_{L^2} &= \sum_i \langle w_i, v\rangle_{L^2} \\
    &= \sum_{i} \langle w_i, \varphi_i v \rangle_{L^2} \\
    &= \sum_i \langle \psi_i u, D_e(\varphi_i v) \rangle_{L^2} \\
    &= \sum_i \langle \psi_i u, D_ev \rangle_{L^2} \\
    &= \langle u, D_ev \rangle_{L^2} .
  \end{align*}
  To get to the second-to-last line we used that the integrand of the inner product $\langle \psi_i u, D_e(\varphi_i v) \rangle_{L^2}$ is non-zero only over the support of $\psi_i$, where $\varphi_i=1$ by assumption. This shows that $u\in\dom(D_e^*)$ with $D_e^*u=w$.

  Conversely, let $u\in\dom(D_e^*)$ be compactly supported. We will show that $u\in\dom(D_d)$. Because $u$ equals the finite sum $\sum_i\psi_i u$, it suffices to show that each $\psi_i u$ is in $\dom(D_d)$. Note that by \cref{LocalFirstOrder}, $D_e^*$ is local because $D_e$ is. Hence $\psi_iu\in\dom(D_e^*)$. As used previously, any element of $\dom(D_e^*)$ satisfies \eqref{loc_adj_cond}. Since we assume \ref{adj_coords}, this implies $\psi_i u\in\dom(D_d)$, whence $u=\sum_i \psi_i u\in\dom(D_d)$. This concludes the proof.
\end{proof}

The following lemma establishes density of compactly supported sections as required for \cref{LocalizeAdjoint}, under an additional boundedness assumption on the principal symbol of $D$. This boundedness assumption is satisfied by any Dirac-type operator. A similar result was established by Bär--Bandara in \cite[Sec.~3]{Baer2025}.

\begin{lem}\label{cpt_dense}
  Let $D$ be a first-order differential operator, and $D_e$ a local boundary condition. Assume that $\olO$ is complete, and that
  $$
    \sup \left\{\left\|\sigma_D(\xi)\right\|\,\colon\, \xi \in T^*\Omega,\,\|\xi\|\leq 1\right\} < \infty .
  $$
  Then, 
  $$
    \left\{ u\in\dom(D_e) \colon \supp(u)\subseteq\olO \mathrm{\; is \; compact} \right\}
  $$
  is dense in $\dom(D_e)$.
\end{lem}

\begin{proof}
  Let $u\in\dom(D_e)$, and fix $\varepsilon>0$. We will find $\psi\in C^\infty_c(\olO)$ such that $\|u-\psi u\|_D<\varepsilon$. To that end, let $K\subseteq \olO$ be a compactum such that $\|u\|_{L^2(\olO\setminus K)}$ and $\|Du\|_{L^2(\olO\setminus K)}$ are less than $\frac{\varepsilon}{4}$. Let $\psi\in C^\infty_c(\olO)$ be such that $\psi=1$ on $K$. Because the principal symbol $\sigma_D$ is assumed bounded and $\olO$ complete, we can further take $\psi$ so that $\|\sigma_D(d\psi)\|\leq 1$\footnote{Let $C < \infty$ be a bound on the pointwise norm of $\sigma_D$. Let $N$ denote the complement of the $3C^{-1}$-neighborhood of $K$, and let $\psi'$ be the function $\psi'(x) = \min\{1, 2C d(x, \olO\setminus N)\}$. It is $(2C)^{-1}$-Lipschitz, compactly supported because $\olO$ is complete, and $\psi'(x)=1$ on the $C^{-1}$-neighborhood of $K$. Then, $\psi'$ can be mollified to a smooth function $\psi\in C^\infty_c(\olO)$ satisfying $\psi|_K=1$ and $\|d\psi\|_\infty \leq C^{-1}$. It follows that $\|\sigma_D(d\psi)\|\leq CC^{-1}= 1$ as claimed.}. Then,
  \begin{align*}
    \|u - \psi u\|_D &= \| u - \psi u\|_{L^2(\olO\setminus K)} + \| D(u - \psi u)\|_{L^2(\olO\setminus K)} \\
    &\leq \| (1-\psi)u\|_{L^2(\olO\setminus K)} + \| Du\|_{L^2(\olO\setminus K)}  + \|\psi Du\|_{L^2(\olO\setminus K)} + \| [\psi,D] u\|_{L^2(\olO\setminus K)} \\
    &\leq \|u\|_{L^2(\olO\setminus K)} + \|Du\|_{L^2(\olO\setminus K)}  + \|Du\|_{L^2(\olO\setminus K)} +  \|\sigma_D(d\psi)\| \cdot \|u\|_{L^2(\olO\setminus K)} \\
    &< \varepsilon ,
  \end{align*}
  which concludes the proof.
\end{proof}

\section{Examples}\label{SectionExamples}

The purpose of this section is to give examples of boundary conditions, and to illustrate their properties. We first relate our notion of regularity to regularity concepts already in use in the literature. Afterwards, we give a few concrete examples.

The theory of elliptic boundary conditions on compact manifolds is classical, see e.g. \cite[Chapter 10]{Hoermander1963}, \cite{BoossBavnbeck1993} or \cite[Section V.11]{Taylor1996}. There, a boundary condition is usually given by an equation of the form $Pu|_{\dO}=0$ for $P$ a \mbox{(pseudo-)} differential operators on the boundary of order less than $D$. Ellipticity is defined via an algebraic condition, the so-called Shapiro--Lopatinski condition, on the symbols of $D$ and $P$. One way of stating this condition is to require the principal symbol $\sigma_P(\xi)$ to be an isomorphism from the sum of eigenspaces of $\sigma_D(\xi)$ with positive real part onto the fiber of the target bundle of $P$ for all non-zero $\xi \in T^*(\partial\Omega)$. Such boundary conditions are semi-regular in the sense of Definition \ref{DefReg}:

\begin{prop}\label{classical_reg}
  Suppose that $\olO$ is compact with smooth boundary, let $D$ be an elliptic differential operator over $\olO$, and let $P$ be a (pseudo-)differential operator over $\dO$ of order less than $D$. Suppose that $(D,P)$ satisfy the Shapiro--Lopatinski condition. Then, 
  \begin{equation}\label{DomDiff}
    \dom(D_P) = \left\{ u \in H^d(\olO) \colon Pu|_{\dO}=0 \right\}
  \end{equation}
  is a semi-regular boundary condition for $D$.
\end{prop}

\begin{proof}
  The Shapiro--Lopatinski condition implies an elliptic regularity estimate for $(D,P)$ (see e.g. \cite[ Prop. V.11.14]{Taylor1996}): If $u \in H^d(\olO)$, then
  \begin{equation}\label{reg_est}
    \|u\|_{H^d(\olO)} \lsim \|u\|_{L^2(\olO)} + \|Du\|_{L^2(\olO)} + \|Pu|_{\dO}\|_{H^{d-\ell-\sfrac{1}{2}}(\dO)} ,
  \end{equation}
  where $\ell < d$ is the order of $P$. For $u\in\dom(D_P)$, this gives $\|u\|_{H^d(\olO)}\leq \|u\|_D$. The converse estimate is a consequence of Sobolev-continuity of $D$. Hence, graph norm and Sobolev norm are equivalent on $\dom(D_P)$, and because $\dom(D_P)$ is closed with respect to the latter, it is also closed with respect to the former. Hence, it actually defines a boundary condition. From there, semi-regularity is apparent.
\end{proof}

\begin{rem}
  The proof of \cref{classical_reg} illustrates that a regularity estimate as in \eqref{reg_est} implies semi-regularity in our sense. Hence, its conclusion extends to other classical regularity notions that come with a regularity estimate, like regularity upon freezing coefficients (also discussed in \cite[Section V.11]{Taylor1996}).
\end{rem}

In recent work \cite{Baer2025}, Bär and Bandara investigate boundary conditions for first-order elliptic operators on manifolds with non-compact, smooth boundary satisfying some mild boundedness conditions. This extends earlier work by Bär and Ballmann treating the case of compact boundaries \cite{BaerBallmann2012}, and by Große and Nakad treating Dirac operators on spin$^c$ manifolds with boundary and bounded geometry \cite{Grosse2014}. Bär and Bandara parametrize boundary conditions via trace maps in terms of closed subspaces of a certain Sobolev-like space of Cauchy-data on the boundary. In their setup, a boundary condition is called semi-regular if its associated subspace is contained in $H^{\sfrac{1}{2}}_\loc(\partial\Omega)$. Continuity of the trace map implies that this is the case if and only if the boundary condition itself is contained in $H^1_\loc(\olO)$.

\begin{prop}[{\cite[Corollary 7.7]{Baer2025}}]
Suppose that $D$ is first-order, and that $\dO$ is smooth. Suppose further that $D$ and $\olO$ satisfy the regularity conditions (M1)--(M5) and (G1)--(G3) of \cite[Section 2]{Baer2025}. Then, the following are equivalent:
\begin{itemize}
\item $D_e$ is (semi-)regular in the sense of Definition \ref{DefReg} above.
\item $D_e$ is (semi-)regular in the sense of \cite[Definition 7.5]{Baer2025}.
\end{itemize}
\end{prop}

In this context, let us mention the related work by Bandara--Goffeng--Saratchandran treating the case of higher-order operators on compact manifolds with boundary \cite{Bandara2023}. They also describe boundary conditions in terms of subspaces of a Cauchy data space on the boundary. Their definition of (semi-)regularity agrees with the one we use, and they give an equivalent condition for (semi-)regularity in terms the associated subspaces.

Manifolds of bounded geometry are a convenient setting for analysis on non-compact manifolds, because by definition they come with uniform quantitative bounds on their geometric structure. This allows one to define global Sobolev spaces that behave essentially analogously to Sobolev spaces on compact manifolds. Furthermore, the theory of uniformly elliptic operators satisfying quantitative bounds on their coefficients and the inverse of their symbols is analogous to the elliptic theory on compact manifolds. Uniformly elliptic boundary conditions for such operators have initially been investigated by Schick in \cite{Schick1998}. More recently, Große and Nistor discussed uniform versions of different classical regularity notions for second-order uniformly elliptic operators \cite{Grosse2020}. These uniformly regular/elliptic boundary conditions are covered by our setting\footnote{In fact, the additional uniformness information allows one to carry out our subsequent $K$-homological discussion of \cref{SecMainResult} in \emph{uniform} K-homology, see \cite{LykoThesis}.}.

\begin{prop}
  Suppose that $\olO$ is a manifold with boundary and bounded geometry, $D$ a uniformly elliptic differential operator over $\olO$, and $P$ a uniform differential operator over the boundary with order less than $P$. Let $D_P$ be defined as in \eqref{DomDiff}, using the global Sobolev space $H^d(\olO)$ on $\olO$.
  \begin{itemize}
    \item[i)] If $(D,P)$ is a uniformly elliptic boundary condition in the sense of \cite[Definition~4.11]{Schick1998}, then $D_P$ is semi-regular.
    \item[ii)] Suppose that $D$ is a second-order differential operator in divergence form. If $(D,P)$ has uniform regularity of the local problems in the sense of \cite[Theorem~5.6]{Grosse2020}, satisfies the uniform Shapiro-Lopatinski condition of \cite[Definition~6.8]{Grosse2020}, or is well-posed in energy space in the sense of \cite[Theorem~7.2]{Grosse2020}, then $D_P$ is semi-regular.
\end{itemize}
\end{prop}

\begin{proof}
  In each case, $(D,P)$ satisfies a uniform regularity estimate (these are \cite[Theorem 4.14]{Schick1998}, and \cite[Theorem 5.6]{Grosse2020}, \cite[Theorem 6.9]{Grosse2020} and \cite[Theorem 7.2]{Grosse2020}, respectively) analogous to \eqref{reg_est}, from which it can be analogously deduced that $\dom(D_P)$ is closed in the graph norm, i.e. a boundary condition. The global Sobolev space $H^d(\olO)$ includes continuously into the local space $H^d_\loc(\olO)$, giving semi-regularity.
\end{proof}

We move on to provide a few concrete examples of boundary conditions and their properties. The goal is by no means a general account, but rather an illustrative one. A summary can be found in Table~\ref{TableExamples}.

\begin{ex}[Minimal boundary condition]\label{ex_minimal_condition}
Every differential operator admits a minimal boundary condition. If the differential operator is symmetric, then the adjoint of the minimal is the maximal boundary condition. The minimal and maximal boundary condition are never equal unless the boundary is empty, therefore neither is self-adjoint.\footnote{In fact, a symmetric differential operator need not admit a self-adjoint boundary conditions at all, and even if it does, there need not be self-adjoint boundary conditions that are also semi-regular, or also local, see e.g. \cite[Example~13.1 and preceeding paragraph]{Schmudgen2012}.} \cref{MinDomSobolev} states that if $D$ is elliptic over $\olO$, then $\dom(D_{\min})$ is continuously included in $H^d_{0,\loc}(\Omega;E)$. Hence, the minimal boundary condition is semi-regular. The minimal boundary condition is also local by \cref{MultMinDomain}. Lastly, if $D$ is odd resp. multigraded, then so is $D_{\min}$.
\end{ex}

\begin{ex}[Maximal boundary condition]\label{ExMinMax}
As discussed in the previous example, the maximal boundary condition is generally not self-adjoint. It is also not regular unless $\olO$ is one-dimensional. For an explicit example, consider the Cauchy-Riemann operator $\bar{\partial}$ on the domain
$$
\Omega:= \left\{ z\in\C \, \colon \, 0<|z|<1, \; 0<\mathrm{arg}(z)<\frac{\pi}{2} \right\} \subset \C \; .
$$
Consider the function $u(z)=z^{-\sfrac{1}{2}}$. It holds that $u\in L^2(\Omega)$, and since $u$ is holomorphic, it holds that $\bar{\partial}u=0$ pointwise. Hence, $u \in \ker(\bar{\partial}_{\max})$, and in particular $u\in\dom(\bar{\partial}_{\max})$. Suppose that $u$ were also in $H^1(\Omega)$. Then, the partial derivative $\partial_x u = -\frac{1}{2} z^{-\sfrac{3}{2}}$ would be in $L^2(\Omega)$, which is not the case.

This example can also be used to illustrate that the maximal boundary condition is not local for higher-order operators (a similar example is given in \cite[~Remark 3.10]{Fries2025}). Since the function $u(z)=z^{-\sfrac{1}{2}}$ is holomorphic over $\Omega$, it is also in the kernel and thus in particular the domain of $\Delta_{\max}$ for the Laplacian $\Delta = \bar{\partial}^*\bar{\partial}$. Let $f\in C^\infty_c(\olO)$ be the function $f(x+iy)=x/2$. Then, 
$$
\Delta(fu) = 2 (\partial_xf)(\partial_x u) = \partial_x u
$$
is not in $L^2(\olO)$. Thus, $\Delta_{\max}$ is not local.

If $D$ is odd resp. multigraded, then so is $D_{\max}$. This follows from \cref{AdjointMultigraded} and the corresponding property of the minimal boundary condition.
\end{ex}

The next three examples concern the Dirac operator. For the sake of concreteness we focus on the spin Dirac operator, though the discussion applies equally to twisted and spin$^c$ Dirac operators. Let us mention the recent work of Große--Uribe--Van Den Bosch giving explicit descriptions of boundary conditions for the Dirac operator under various locality, regularity and self-adjointness assumption on low-dimensional manifolds \cite{Grosse2025}, though their work also contains results in general dimension.  

Assume that $\olO$ is a complete spin manifold with smooth boundary. Let $\slashed{S}(\olO)\to \olO$ denote the complex spinor bundle, and $\Dirac\colon \Gamma_{cc}(\Omega;\slashed{S}(\olO))\to \Gamma_{cc}(\Omega;\slashed{S}(\olO))$ the spin Dirac operator. If the dimension of $\Omega$ is odd, then $\slashed{S}(\olO)|_{\partial \Omega}$ is naturally isomorphic to $\slashed{S}(\dO)$, and if the dimension is even, then it has two graded components, both of which are isomorphic to $\slashed{S}(\dO)$. If $X$ is a tangent vector (field), we will denote by $c(X)$ the corresponding Clifford multiplication operator on $\slashed{S}(\olO)$. Lastly, we fix an inward-pointing unit normal $\nu$.

\begin{ex}[Chiral boundary conditions]
Clifford multiplication by $i\nu$ is a self-adjoint involution on $\slashed{S}(\olO)|_{\partial \Omega}$. Let $P_\pm$ denote projection onto its $(\pm 1)$-eigenspace, and define the \emph{chirality boundary conditions} $\Dirac_\pm$ by\footnote{We should elaborate how the condition $P_\pm u|_{\partial\Omega} = 0$ is to be interpreted. For general $\olO$ it is not clear if the maximal boundary condition $\dom(\Dirac_{\max})$ admits a trace map $u\mapsto u|_{\dO}$. Such trace maps do exist locally (for example, transfer coordinate (half-)balls to compact manifolds, and use the results of \cite{BaerBallmann2012}). Then, we can define $\dom(\Dirac_\pm)$ as the set of those $u\in\dom(\Dirac_{\max})$ which satisfy $P_\pm(\psi u)|_{\dO}=0$ for all $\psi\in C^\infty_c(\olO)$. Because the local trace maps are continuous, this defines a closed subset of $\dom(\Dirac_{\max})$, and hence a boundary condition. Moreover, since $P_\pm$ are bundle morphisms, hence commute with functions, this definition coincides with the expected interpretation of $P_\pm u|_{\dO}=0$ whenever a global trace $u\mapsto u|_{\dO}$ does exist.}
$$
\dom(\Dirac_\pm) = \left\{ u\in \dom(\Dirac_{\max}) \colon P_\pm u|_{\partial\Omega} = 0 \right\} .
$$
Since they are defined by bundle morphisms, the chirality boundary conditions are local. In \cref{RegAbsRelChiral} below we show that $\Dirac_\pm$ are regular, and $\Dirac_\pm^*= \Dirac_\mp$. However, the chiral boundary conditions are not graded. Indeed, suppose that $\dim(\Omega)$ is even. Then the spinor bundle is graded by the complex volume element $\kappa$, which is locally given by $\kappa=\pm i c(e_1)\cdots c(e_{n-1})c(\nu)$. Then, $\kappa c(\nu) = - c(\nu)\kappa$. Thus the grading operator $\kappa$ permutes the kernels of $P_\pm$, and we conclude that $\kappa \Dirac_\pm = -\Dirac_\mp \kappa$.
\end{ex}

\begin{ex}[APS-boundary condition]
For this example, we assume conditions (A1)-(A3) from \cite[Section~8.1]{Baer2025} to be able to use their results on boundary conditions for Dirac operators, in particular the existence of a global trace map. These conditions roughly require the existence of a uniformly thick collar neighborhood, on which curvature and second fundamental form are bounded. In particular, they hold for manifolds with compact boundary, or manifolds with boundary and bounded geometry. 

On the boundary, $\Dirac$ takes the form
$$
\Dirac = c(\nu) \left( \partial_\nu + A \right)  \; ,
$$
where $c(\nu)A=\Dirac_{\partial\Omega}$ if $\dim(\Omega)$ is odd, and $c(\nu)A=\begin{pmatrix}
0 & \Dirac_{\partial\Omega} \\ \Dirac_{\partial \Omega} & 0
\end{pmatrix}$ if it is even. In either case $A$ is an essentially self-adjoint operator, and thus admits a spectral projection $\pi_{<}=\mathbf{1}_{(-\infty,0)}(A)$. The \emph{APS-boundary condition} is defined by
$$
\dom(\Dirac_\mathrm{APS}) = \left\{ u\in\dom(\Dirac_{\max}) \colon \pi_{<}u|_{\partial\Omega} = 0 \right\}
$$
The spectral projection $\pi_{<}$ is a \emph{non-local} operator, therefore $\Dirac_\mathrm{APS}$ is not local. However, it is regular \cite[Proposition 7.23]{Baer2025}. The domain of $\dom(\Dirac_\mathrm{APS}^*)$ equals $\dom(\Dirac_\mathrm{APS}) \oplus \ker(A)$ \cite[Remark 7.24]{Baer2025}, so that the APS-boundary condition is self-adjoint if and only if $A$ has trivial kernel, which in turn is the case if and only if $\Dirac_{\dO}$ has trivial kernel. Moreover, if $\olO$ is even-dimensional, then the spinor bundle is graded by the complex volume element, which locally can be written as $\kappa=\pm i c(e_1)\cdots c(e_{n-1}) c(\nu)$, where $e_1,\ldots,e_{n-1}$ form a local orthonormal frame of $\nu^\perp=T(\dO) \subseteq T\olO$. The operator $\Dirac$ is odd with respect to this grading. Since both $c(\nu)$ and the operator $\begin{pmatrix}
0 & \Dirac_{\partial\Omega} \\ \Dirac_{\partial \Omega} & 0
\end{pmatrix}$ are odd with respect to this grading, the operator $A$ is commutes with the grading operator $\kappa$. Hence the same holds for $\pi_{<}$, and we conclude that $\kappa$ preserves $\dom(\Dirac_\mathrm{APS})$. Hence the APS-condition is odd in even dimensions. 
\end{ex}

\begin{ex}[Transmission boundary condition]\label{ex_transmission}
  Let $M$ be a complete spin manifold, and $\Sigma$ a hypersurface. We assume for simplicity that $\partial M =\emptyset$, and that $\Sigma$ is compact. We form the manifold $\olO$ by removing a tubular neighborhood of $\Sigma$ from $M$. The boundary $\dO$ has two connected components, $\Sigma_1$ and $\Sigma_2$, each can be identified with $\Sigma$, though they carry opposite orientations. The \emph{transmission boundary condition} for $\Dirac_\Omega$ is defined by
  $$
    \dom(\Dirac_T) = \left\{ (u,u) \in H^{\sfrac{1}{2}}(\Sigma_1) \oplus H^{\sfrac{1}{2}}(\Sigma_2) \colon u\in H^{\sfrac{1}{2}}(\Sigma)\right\} \subseteq H^{\sfrac{1}{2}}(\dO) .
  $$
  It is graph-norm closed, hence indeed a boundary condition \cite[Ex.~7.28]{BaerBallmann2012}. It is also self-adjoint and regular \cite[~Theorem 7.4]{Grosse2025}. It is not local: A function $f\in C^\infty_c(\olO)$ preserves $\dom(\Dirac_T)$ if and only if $f|_{\Sigma_1}=f|_{\Sigma_2}$. In even dimensions, $\Dirac_T$ is odd.
\end{ex}

For the next example, again assume that $\olO$ has smooth boundary. We consider $d+d^*$ acting on the bundle $\bigwedge\olO=\bigoplus_{p=0}^{\dim(\Omega)} \bigwedge^p T^*\olO$, graded by even/odd degree. The resulting odd, graded operator is the Euler characteristic operator, which we denote by $D^\chi$. Again, fix an inward-pointing unit normal $\nu$.

\begin{ex}[Absolute and relative boundary conditions for $D^\chi$]\label{ex_defn_euler_op}
There is an orthogonal splitting
$$
\left(\bigwedge  T^*\olO\right)\bigg|_{\partial \Omega} = \left( \bigwedge T^*(\partial \Omega)  \right) \oplus \left( \bigwedge T^*(\partial \Omega)  \right) \wedge (\C\cdot \nu^\flat).
$$
Denote by $P_\mathrm{tan}$ the projection onto the first summand, and by $P_\mathrm{nor}$ that onto the second. Define the \emph{absolute} and \emph{relative} boundary conditions $D^\chi_\mathrm{abs}$ and $D^\chi_\mathrm{rel}$ by\footnote{Again, in the potential absence of a global trace map, we can make sense of the condition $P_j\alpha|_{\dO} = 0$ locally.}
\begin{align*}
\dom(D^\chi_\mathrm{abs}) &= \left\{ \alpha\in\dom(D^\chi_{\max}) \, \colon \, P_\mathrm{nor}\alpha|_{\dO} = 0 \right\}, \\
\dom(D^\chi_\mathrm{rel}) &= \left\{ \alpha\in\dom(D^\chi_{\max}) \, \colon \, P_\mathrm{tan}\alpha|_{\dO} = 0 \right\}  .
\end{align*}
They are defined by bundle morphism, and hence local. Moreover, they are both self-adjoint and regular, see the \cref{RegAbsRelChiral} below. The grading by degree preserves the domains of $D^\chi_\mathrm{abs/rel}$, hence $D^\chi_\mathrm{abs/rel}$ is odd.
\end{ex}

\begin{ex}[Dirichlet and Neumann Laplacian]\label{ExNeumann}
Consider the Laplace operator $\Delta=d^*d$ on functions and fix again an inward-pointing unit normal $\nu$. Define the \emph{Dirichlet and Neumann boundary conditions} by
\begin{align*}
\dom(\Delta_\mathrm{Dir}) &= \left\{ u\in\dom(\Delta_{\max}) \, \colon \, u|_{\partial\Omega} = 0 \right\} \\
\dom(\Delta_\mathrm{Neu}) &= \left\{ u\in\dom(\Delta_{\max}) \, \colon \, (\partial_\nu u)|_{\partial\Omega} = 0 \right\} \; .
\end{align*}
In the case that $\olO$ is a manifold with boundary and bounded geometry, it is known that both are self-adjoint \cite[Theorem 1.2]{Ammann2019} and regular \cite[Proposition 5.9]{Schick1998}. The Dirichlet condition is local in the sense of \cref{def_local}, but the Neumann condition is not. Indeed, a function $f\in C^\infty_c(\olO)$ preserves $\dom(\Delta_\mathrm{Neu})$ if and only if $(\partial_\nu f)|_{\partial \Omega}=0$. This illustrates that differential boundary conditions need not be local in our sense. Indeed, if the differential operator defining the boundary condition has order at least $1$, there is no reason to expect the boundary condition to be local in this sense.
\end{ex}

\begin{table}[t]
\centering
\begin{tabular}{|c||c|c|c|c|c|c|}
\hline
  & self-adjoint & local & semi-reg. & reg. & odd & multigraded \\ \hline \hline
 $D_{\min}$ & x & $\checkmark$ & $\checkmark$ & x & ($\checkmark$) & ($\checkmark$) \\  \hline
 $D_{\max}$ & x & If $d=1$ & x & x & ($\checkmark$) & ($\checkmark$) \\  \hline
  $\Dirac_{\pm}$ & $\Dirac_\pm^*=\Dirac_\mp$ & $\checkmark$ & $\checkmark$ & $\checkmark$ & ($\kappa\Dirac_\pm=-\Dirac_\mp\kappa$) &  \\  \hline
 $\Dirac_\mathrm{APS}$ & Iff $\ker(\Dirac_{\partial\Omega})=0$ & x & $\checkmark$ & $\checkmark$ & ($\checkmark$) &  \\  \hline
  $\Dirac_T$ & $\checkmark$ & x & $\checkmark$ & $\checkmark$ & ($\checkmark$) &  \\  \hline
  $D^\chi_\mathrm{abs/rel}$ & $\checkmark$ & $\checkmark$ & $\checkmark$ & $\checkmark$ & $\checkmark$ &  \\  \hline
  $\Delta_\mathrm{Dir}$ & $\checkmark$ & $\checkmark$ & $\checkmark$ & $\checkmark$ & & \\ \hline
  $\Delta_\mathrm{Neu}$ & $\checkmark$ & x & $\checkmark$ & $\checkmark$ & & \\ \hline
\end{tabular}
\caption{Summary of the properties of the boundary conditions discussed in this section. A check-mark in brackets indicates that the property is possessed if the prerequisites are met. For example the mark ($\checkmark$) for $D_{\min}$ in the \glqq odd\grqq -column means that if $D$ is odd, then so is $D_{\min}$. Similarly, no mark is made if the grading property cannot meaningfully applied to the operator to begin with. For example the box for $\Delta_\mathrm{Dir}$ in the \glqq odd\grqq -column is left empty because the Laplacian is not a graded operator.}\label{TableExamples}
\end{table}

We now provide proofs for the regularity and adjoints of absolute/relative and chiral boundary conditions on general non-compact manifolds with smooth boundary. For compact manifolds, these properties are well-established (see e.g. \cite[Lemma 4.1.1]{Gilkey1996} and \cite[Section 5]{Hijazi2002}), and our proof works by reduction to that case. In the non-compact case, \cite{Baer2025} provide general techniques for establishing these properties, which work under their mild boundedness assumptions.

\begin{lem}\label{RegAbsRelChiral}
  Let $\olO$ be a complete Riemannian manifold with smooth boundary.
  \begin{enumerate}
    \item\label{AbsRelReg} The absolute and relative boundary conditions for the Euler characteristic operator over $\olO$ are self-adjoint and regular.
    \item\label{ChiralReg} Suppose that $\olO$ is spin. The chiral boundary conditions for the spin Dirac operator over $\olO$ are regular, and $\Dirac^*_\pm = \Dirac_\mp$.   
  \end{enumerate}
\end{lem}

\begin{proof}
  We begin with \ref{AbsRelReg}, and treat the absolute boundary condition for concreteness. We will prove regularity and self-adjointness via reduction to the compact case, where these properties are already known to hold (see e.g. \cite[Lemma 4.1.1]{Gilkey1996}). The absolute boundary condition is local. By \cref{RegLocalized} it therefore suffices to prove that for every coordinate (half-)ball $V\subseteq \Omega^\circ$ that is precompact in $\olO$, every $\psi\in C_{c}^\infty(\olO)$ with support in $\overline{V}$, and all $u\in\dom(D_\mathrm{abs}^\chi)$ it holds that $\psi u \in H^1(V;\bigwedge T^*V)$ with estimates
  $$
    \|\psi u \|_D \lsim \|\psi u \|_{H^1(V;\bigwedge T^*V)} \lsim \|\psi u \|_D .
  $$
  Using a coordinate map, transfer $V$ into $\R^n$. There, choose a compact domain $\overline{U}\subset \R^n$ with smooth boundary that contains $V$ in such a way that $\overline{V} \cap \partial U$ is equal to the image of $\overline{V} \cap \dO$ under the coordinate map. Over $V$ the Euler characteristic operator $D^\chi$ of $\Omega$ (pulled back via the coordinate map) coincides with the Euler characteristic operator $D^{\chi, U}$ over $U$. Moreover, $u\in\dom(D^\chi_\mathrm{abs})$ implies that $\psi u \in \dom(D^{\chi,U}_\mathrm{abs})$, since the boundary conditions for both operators are defined locally via an inward-pointing unit normal. Now, the absolute boundary condition on a compact manifold with smooth boundary is semi-regular. Hence, $\psi u \in H^1(V;\bigwedge T^*V)$ with the corresponding norm estimate. This establishes the condition \ref{cond_coords} of \cref{RegLocalized}, and thus semi-regularity of the absolute boundary conditions.

  To see that $D^\chi_\mathrm{abs}$ is self-adjoint (and therefore regular), we appeal to \cref{LocalizeAdjoint}, which is applicable due to \cref{cpt_dense}. Thus, it suffices to verify property \ref{adj_coords} from \cref{LocalizeAdjoint} for every coordinate half-ball $V$. Again, we may transfer such a half-ball to a compact smooth domain $U$ in Euclidean space, and treat the absolute boundary condition $D^{\chi,U}_\mathrm{abs}$ of that domain. On compact manifolds with smooth boundary the absolute boundary condition is known to be self-adjoint. \cref{LocalizeAdjoint} in particular implies that \ref{adj_coords} holds for the specific half-ball $V$ in question. Since the absolute boundary conditions for $D^\chi$ and $D^{\chi,U}$ are identified over $V$ via the coordinate map, we conclude that \ref{adj_coords} holds for the absolute boundary condition, whence it is self-adjoint. This proves the claimed properties of $D^\chi_\mathrm{abs}$. The exact same proof applies to $D^\chi_\mathrm{rel}$. 

  The chiral boundary conditions are handled analogously by a reduction to compact smooth domains in Euclidean space, which is equipped with the spin structure induced by the ambient Euclidean space. For these boundary conditions regularity and self-adjointness in the compact case is shown in \cite[Section 5]{Hijazi2002}.
\end{proof}

\section{The formal operator calculus}\label{SectionOpCalc}

On manifolds without boundary, the prerequisite (local) compactness statements needed to prove that a given elliptic operator defines a K-homology cycle are conveniently proved via mapping properties of pseudo-differential operators on the Sobolev scale. To replicate that style of argument in the presence of a boundary, we will use a formal scale of spaces adapted to the given boundary condition, and investigate the mapping properties of operators on this scale. This circumvents pseudo-differential theory, and the need to establish the required Sobolev estimates. Only after carrying out the computations on the formal scale will we appeal to (semi-)regularity to connect our formal scale to the Sobolev scale. Nevertheless, the reader is invited to mentally connect the spaces $\cH^\alpha$ introduced below with Sobolev spaces, and operators mapping $\cH^\alpha$ to $\cH^{\alpha+\beta}$ with operators that are smoothing of degree $\beta$.

\subsection{The abstract framework}

We begin with a general discussion of the scale associated to a fixed closed operator $T$ on a Hilbert space $H$. For $\alpha\geq0$ let
\[
\cH^\alpha(T) \coloneqq \text{domain of } (T^*T+I)^{\frac{\alpha}{2}} \text{ equipped  with } \|u\|_{\cH^\alpha(T)} \coloneqq \|(T^*T+I)^{\alpha/2} u\|.
\]
This norm is equivalent to the graph norm of $(T^*T+I)^{\frac{\alpha}{2}}$. Moreover, we set
\[
\cH^{-\alpha}(T) \coloneqq (\cH^\alpha(T))^*.
\]
These are Hilbert spaces, and for $\alpha\leq\alpha'$ we have a bounded inclusion $\cH^{\alpha'}(D_e)\hookrightarrow \cH^\alpha(D_e)$. Lastly, we set
\[
\cH^\infty(T) \coloneqq \bigcap_{\alpha\in\R}\cH^\alpha(T) \quad\text{and}\quad \cH^{-\infty}(T) \coloneqq \bigcup_{\alpha\in\R}\cH^\alpha(T).
\]
When the operator $T$ is clear from context we shall take the liberty of writing $\cH^\alpha$, $\cH^\infty$, etc. instead of $\cH^\alpha(T)$, $\cH^\infty(T)$, etc.
\begin{prop}\label{FormalSobolevStructure}
\mbox{}
\begin{enumerate}
\item\label{item_first_scale} It holds that $\cH^1(T)=\dom(T)$ with equivalent norms.
\item For all $\alpha\geq 0$ the inner product on $H$ extends to a perfect pairing between $\cH^\alpha(T)$ and $ \cH^{-\alpha}(T)$.
\item For all $\alpha,\beta\in\R$, $\alpha\leq \beta$, the embedding $\cH^\beta(T)\hookrightarrow\cH^\alpha(T)$ has dense image. In fact, $\cH^\infty(T)$ is dense in each $\cH^\alpha(T)$. 
\item\label{item_scale_norm} For $\alpha \geq 0$ the norm of $u\in H$ as an element of $\cH^{-\alpha}(T)$ is given by 
$$
\|u\|_{\cH^{-\alpha}(T)} =  \left|\left| \left(T^*T+I\right)^{-\frac{\alpha}{2}}u \right|\right| .
$$
\item\label{item_interpolation_scale} For all $\alpha < \beta$ the interpolation space $[\cH^\alpha(T),\cH^\beta(T)]_\theta$ is given by $\cH^{(1-\theta)\alpha+\theta\beta}(T)$.
\end{enumerate}
\end{prop}
\begin{proof}
Items \ref{item_first_scale}--\ref{item_scale_norm} are proved in \cite[Appendix A.1]{Baer2025} and Item \ref{item_interpolation_scale} is proved in \cite[Proposition 4.2.2]{Taylor1996}.\footnote{Strictly speaking, in that proposition it is only proved that $[H,\cH^k]_\theta=\cH^{k\theta}$. This takes care of the case $\alpha\geq 0$.For $\beta <0$ note that via the dense embedding $H \hookrightarrow \cH^{\alpha}$, $\dom(T)=\dom((T^*T+I)^{1/2})$ is also the domain of the operator $T^{(\alpha)} \coloneqq (T^*T+I)^{\frac{1-\alpha}{2}}$ viewed as an unbounded operator on $\cH^\alpha$. In fact, since $(T^*T+I)^{-\alpha/2}$ is an isometric isomorphism from $H$ to $\cH^\alpha$, the operator $T^{(\alpha)}$ is closed, and it holds that $\cH^{\gamma}(T^{\alpha})=\cH^{(1-\alpha)\gamma+\alpha}(T)$. The case $\alpha <0$ may thus be reduced to the case $\alpha=0$.}
\end{proof}
Let $A\colon \cH^\infty(T)\to \cH^{-\infty}(T)$ be a linear map. The majority of our investigations will revolve around the question whether there is $\beta\in\R$ such that $A$ extends to a bounded map $\cH^\alpha(T)\to \cH^{\alpha-\beta}(T)$ for some or even all $\alpha$. In the case where $T=D_e$ is a suitable boundary condition, such operators have similar mapping properties as pseudo-differential operators on the Sobolev scale, so we should think of them as operators of order $\beta$ on the scale associated to $T$. Our first lemma is an application of duality and interpolation. 
\begin{lem}\label{DualityInterpolation}
Let $A,B\colon \cH^\infty \to H$ be linear maps with $\langle u,Av \rangle_H= \langle Bu,v\rangle_H$. Let $\alpha,\beta\in \R$. 
\begin{enumerate}
\item\label{DualityInterpolation_a} If $A$ extends to a bounded map from $\cH^\alpha$ to $\cH^{\alpha-\beta}$, then $B$ extends to a bounded map from $\cH^{-\alpha+\beta}$ to $\cH^{-\alpha}$.
\item\label{DualityInterpolation_b} If $A$ and $B$ extend to bounded maps $\cH^{\alpha} \to \cH^{\alpha-\beta}$, then $A$ and $B$ extend to bounded maps $\cH^{\gamma}\to \cH^{\gamma-\beta}$ for all $\gamma$ between $-\alpha+\beta$ and $\alpha$.
\end{enumerate}
\end{lem}
\begin{proof}
Ad \ref{DualityInterpolation_a}: $\cH^{-\alpha}$ is the dual space of $\cH^\alpha$, by definition if $\alpha\geq 0$ and by reflexivity if $\alpha<0$. Consider the dual operator $A'$ to $A \colon \cH^\alpha \to \cH^{\alpha-\beta}$, it maps $(\cH^{\alpha-\beta})^*=\cH^{-\alpha+\beta}$ to $(\cH^{\alpha})^*=\cH^{-\alpha}$. By \cref{FormalSobolevStructure} it is fully determined by the images of $\varphi_u \in (\cH^{\alpha})^*$, $\varphi_u(v)=\langle u,v\rangle_H$, for $u\in\cH^\infty\subseteq \cH^{\alpha-\beta}$. For $v\in \cH^\infty\subseteq \cH^{-\alpha+\beta}$ this image is given by
$$
\left(A'\varphi_u \right) v = \varphi_u(Av) = \langle u,Av \rangle_H = \langle Bu,v \rangle_H = \varphi_{Bu}(v) .
$$
Note that since $B$ extends to a bounded map $\cH^\alpha \to \cH^{\alpha-\beta}$, $\varphi_{Bu}$ is an element of the dual of $\cH^{-\alpha+\beta}$. Since $A'$ is bounded as the dual operator to a bounded operator, it follows that $B$ extends to a bounded operator $\cH^{-\alpha+\beta}\to \cH^{-\alpha}$ via $\varphi_u \mapsto \varphi_{Bu}$. 

Ad \ref{DualityInterpolation_b}: If $A,B$ both extend to bounded maps from $\cH^\alpha$ to $\cH^{\alpha-\beta}$, then both maps also extend to bounded maps from $\cH^{-\alpha+\beta}$ to $\cH^{-\alpha}$ by item (i). Therefore $A$ and $B$ also extend to bounded maps from $[\cH^{-\alpha+\beta},\cH^\alpha]_\theta$ to $[\cH^{-\alpha},\cH^{\alpha-\beta}]_\theta$ by interpolation, and by \cref{FormalSobolevStructure} these exhaust all intermediate exponents.
\end{proof}
Next we come to functions of the operators. Applying a function $f$ to the operator $T$ produces an operator of order $\beta$ on the scale of $T$, where the order $\beta$ depends on the growth of the function $f$. Concretely, for $\beta \in\R$ let $\cF^\beta$ denote the space of measurable functions $f$ on $\IR$ for which $(1+x^2)^{-\frac{\beta}{2}} f \in \mathcal{L}^\infty(\IR)$.
We equip $\cF^\beta$ with the norm $\|f\|_{\cF^\beta} \coloneqq \|(1+x^2)^{-\frac{\beta}{2}} f\|_\infty$. Then, we have the following:
\begin{prop}\label{MappingPropertiesT}
Let $\alpha,\beta\in\R$ and $f\in\cF^\beta$. 
\begin{enumerate}
\item\label{MappingPropertiesT_a} $T$ is a bounded operator from $\cH^\alpha(T)$ to $\cH^{\alpha-1}(T^*)$, and $T^*$ is a bounded operator from $\cH^\alpha(T^*)$ to $\cH^{\alpha-1}(T)$.
\item\label{MappingPropertiesT_b} $f(|T|)$ defines a bounded operator from $\cH^\alpha(T)$ to $\cH^{\alpha-\beta}(T)$, and $\|f(|T|)\|\leq\|f\|_{\cF^\beta}$ independently of $\alpha$.
\item\label{MappingPropertiesT_c} Suppose that $T$ is self-adjoint. Then $f(T)$ defines a bounded operator from $\cH^\alpha(T)$ to $\cH^{\alpha-\beta}(T)$, and $\|f(T)\|\leq \|f\|_{\cF^\beta}$ independently of $\alpha$.
\end{enumerate}
\end{prop}
\begin{proof}
Ad \ref{MappingPropertiesT_a}: Let $u\in \cH^\alpha(T)\cap \cH^1(T)$. Then,
$$
\|Tu\|_{\cH^{\alpha-1}(T^*)} = \|(TT^*+I)^{\frac{\alpha-1}{2}} Tu\| = \| T(T^*T+I)^{\frac{\alpha-1}{2}} u \| \leq \| (T^*T+I)^{\frac{\alpha}{2}} u \| \leq \|u\|_{\cH^\alpha(T)} ,
$$
for any $\alpha\in\R$, where we have used that $Tf(T^*T)=f(TT^*)T$ for $f\in C_0(\R)$, see the proof of \cite[Lemma~5.8(ii)]{Schmudgen2012}. Using that $\cH^1$ is dense in $\cH^\alpha(T)$ for $\alpha\leq 1$ the claim for $T$ follows. For $T^*$ it is proved analogously.

Ad \ref{MappingPropertiesT_b}: By definition of $\cF^\beta$, the operator $(T^*T+I)^{-\beta/2}f(|T|)$ is bounded, and it has norm equal to $\|f\|_{\cF^\beta}$. For $u\in \cH^\infty(T)$ we conclude
$$
\|f(|T|)u\|_{\cH^{\alpha-\beta}} = \left|\left| \left((T^*T+I)^{-\beta/2} f(|T|) \right)\cdot (T^*T+I)^{\alpha/2} u \right|\right| \leq \|f\|_{\cF^\beta}\cdot \|u\|_{\cH^\alpha(T)} .
$$
Using that $\cH^\infty(T)$ is dense in $\cH^\alpha(T)$, the claim for $f(|T|)$ follows.

The proof of \ref{MappingPropertiesT_c} is completely analogous to \ref{MappingPropertiesT_b}.
\end{proof}
\subsection{The case of elliptic operators}\label{SubsecFormalElliptic}
In our application to K-homology, we shall need compactness statements related to boundary conditions for elliptic operators. These must come from the corresponding statements for Sobolev spaces. We will handle calculations with boundary conditions with the formal operator calculus. Therefore, we need to connect the formal scale of these boundary conditions to the Sobolev scale, so that we can then deduce compactness statements based on behavior in the operator calculus. More concretely, we will see how statements that operators have negative order in the abstract operator calculus of a boundary condition imply compactness statements. 

Throughout this section we consider two situations simultaneously: The behavior of arbitrary boundary conditions in the interior, and that of semi-regular boundary conditions up to the boundary. In either case we can connect the domain of the boundary condition to Sobolev space, in the former case by interior regularity and in the latter by assumed semi-regularity. In either case we will deduce the desired statements for the operator calculus from the corresponding statement for Sobolev spaces. The behavior of arbitrary boundary conditions on the interior will eventually give rise to the fact that any boundary condition defines a K-homology class on the interior, while that of (local) semi-regular boundary conditions will yield the relative K-homology classes associated to such boundary conditions, or even absolute classes if regularity and self-adjointness are assumed.

Let $M$ be a smooth Riemannian manifold, $\Omega \subseteq M$ a Lipschitz domain, and $E,F\to \Omega$ vector bundles. Let $D\colon \Gamma_{cc}(\Omega;E)\to\Gamma_{cc}(\Omega;E)$ be a differential operator of order $d$ that is elliptic over $\olO$. Let $D_e$ be a boundary condition for $D$. Then, $D_e$ is a closed operator on the Hilbert space $L^2(\Omega)$, and we will investigate its formal scale $\cH^\alpha(D_e)$. We begin with a lemma concerning multiplication by functions.
\begin{lem}\label{LemmaMultFunc}
Let $D_e$ be a boundary condition for $D$, and $f\in C^\infty_c(\olO)$. Suppose that either $D_e$ is arbitrary and $\supp(f)\subset\Omega^\circ$, or $D_e$ is semi-regular and $f\cdot\dom(D_e)\subseteq \dom(D_e)$. Then multiplication by $f$ is a bounded map 
$$
\cH^{\alpha}(D_e) \longrightarrow \cH^{\alpha}(D_e)
$$
for all $0\leq\alpha\leq 1$.
\end{lem}
\begin{proof}
For $\alpha=1$ the statement holds by Proposition \ref{MultInt} or Lemma \ref{MultSemiReg}, for $\alpha=0$ it is true because $\cH^0(D_e)=L^2(\Omega;E)$. For intermediate degrees it follows by interpolation via \cref{FormalSobolevStructure}.\ref{item_interpolation_scale}.
\end{proof}

\begin{lem}\label{lem_regular_interpolation}
Let $D_e$ be a semi-regular boundary condition for $D$. Then for all $0 \le \alpha \le 1$ we have a continuous inclusion 
$$
\cH^\alpha(D_e) \hookrightarrow H^{d\alpha}_\loc(\Omega;E) .
$$
\end{lem}

\begin{proof}
We have to show that for each Lipschitz domain $V\subseteq\Omega^\circ$ that is precompact in $\olO$ and each $u\in\cH^\alpha(D_e)$ it holds that $u|_V \in H^{d\alpha}(V;E)$ with $\|u\|_{H^{d\alpha}(V;E)}\lsim \|u\|_{\cH^\alpha(D_e)}$. This is true for $\alpha=1$ due to the assumed semi-regularity, and the fact that $\dom(D_e)=\cH^1(D_e)$ with equivalent norms, see \cref{FormalSobolevStructure}.\ref{item_first_scale}. Hence, restriction to $V$ yields a continuous map
$$
  \cH^1(D_e) \hookrightarrow H^d_\loc(\Omega;E) \longrightarrow H^d(V;E) .
$$
Interpolation of this map together with \cref{InterpolationSobolev} yields a continuous map 
$$
  \cH^\alpha(D_e) \longrightarrow H^{d\alpha}(V;E) ,
$$
which on elements is still simply restriction to $V$. This shows that for each $u\in\cH^\alpha(D_e)$ we have $u|_V \in H^{d\alpha}(V;E)$ with $\|u\|_{H^{d\alpha}(V;E)}\lsim \|u\|_{\cH^\alpha(D_e)}$, proving the claim.
\end{proof}
For compact $K\subseteq \olO$ we shall write
\[
L^2_K(\Omega;E) \coloneqq \{u\in L^2(\Omega;E)\colon \mathrm{ess\text{-}supp}(u) \subset K\}.
\]
In the following we transfer the Rellich--Kondrachov theorem to the spaces $\cH^\alpha(D_e)$. This forms the technical backbone of our subsequent investigations into the K-homology classes of elliptic operators.
\begin{prop}\label{FormalRellich}
Let $K\subseteq \olO$ be a compact subset. Assume that either:
\begin{enumerate}
\item\label{FormalRellich_a} $K\subset \Omega^\circ$, and $D_e$ is arbitrary,
\item\label{FormalRellich_b} or $K$ is arbitrary, and $D_e$ is semi-regular.
\end{enumerate}
Let $0<\alpha\leq 1$. The maps 
$$
\cH^\alpha(D_e)\cap L^2_K(\Omega;E) \xlongrightarrow{j_\alpha} L^2_K(\Omega;E) \quad \mathrm{and} \quad L^2_K(\Omega;E) \xlongrightarrow{j_{-\alpha}} \cH^{-\alpha}(D_e) \cap L^2_K(\Omega;E)
$$
are compact.
\end{prop}
\begin{proof}
We begin with analysis of the map $j_\alpha$. For the case \ref{FormalRellich_a} fix a smooth, compactly supported function $\psi\in C_\cc^\infty(\Omega^\circ)$ with $\psi|_K=1$. Multiplication by $\psi$ is a continuous map from $\dom(D_e)$ to $\dom(D_{\min})$ by \cref{MultInt}, hence interpolates to a continuous map from $\cH^\alpha(D_e)$ to $\cH^\alpha(D_{\min})$. Hence, case \ref{FormalRellich_a} reduces to the minimal boundary condition. Since $D_{\min}$ is semi-regular, case \ref{FormalRellich_a} further reduces to case \ref{FormalRellich_b}.

Let us discuss case \ref{FormalRellich_b} for the map $j_\alpha$. Hence, let $D_e$ be semi-regular, and $K\subseteq \olO$ an arbitrary compactum. By \cref{lem_regular_interpolation} we may factor $j_\alpha$ as 
$$
\cH^\alpha(D_e) \cap L^2_K(\Omega;E) \longrightarrow H^{d\alpha}_\loc(\Omega;E) \cap L^2_K(\Omega;E) \subseteq H^{d\alpha}(V;E)\longrightarrow L^2(V;E) ,
$$
where $V\subseteq\Omega^\circ$ is some Lipschitz domain that is precompact in $\olO$ and contains $K\setminus\partial M$. The last arrow of that composition is compact by \cref{prop_rellich_kondrachov}, hence so is $j_\alpha$.

For the compactness of $j_{-\alpha}$ simply note that the dual space of $\cH^\alpha(D_e)\cap L^2_K(\Omega;E)$ is isomorphic to 
$$
\cH^{-\alpha}(D_e)/\{ u \in \cH^{-\alpha} \colon u|_K = 0\} \cong \left\{ u\in \cH^{-\alpha}(D_e) \colon \supp(u)\subseteq K \right\},
$$
which contains $\cH^{-\alpha}(D_e) \cap L^2_K(\Omega;E)$ as a subspace. Under this isomorphism $j_{-\alpha}$ corresponds to the dual map of $j_\alpha$. Since the dual of a compact map is compact, this concludes the proof.
\end{proof}
Using this adaptation of the Rellich--Kondrachov theorem, we deduce that if is locally smoothing on the formal scale of a boundary condition, then it is locally compact:
\begin{prop}\label{NegOrderCpt}
Let $D_e$ be a boundary condition for $D$. Let $f\in C^\infty_c(\olO)$ be a smooth, compactly supported function, and suppose that either
\begin{itemize}
  \item[(a)] $D_e$ is arbitrary, and $f$ is supported inside $\Omega^\circ$,
  \item[(b)] or $D_e$ is semi-regular, and $f$ is arbitrary.
\end{itemize} 
Let $T$ be a bounded operator on $L^2(\Omega;E)$ such that $Tf$ extends to a bounded operator $\cH^{-\alpha}(D_e) \to L^2(\Omega;E)$ for some $\alpha>0$. Then $Tf$ is a compact operator on $L^2(\Omega;E)$.
\end{prop}
\begin{proof}
Let $K$ denote the support of $f$. Since $Tf=0$ on the complement of $L^2_K(\Omega;E)$, it suffices to prove that the restriction of $Tf$ to $L^2_K(\Omega;E)$ is compact. This restriction we can factor as
$$
  L^2_K(\Omega;E) \xlongrightarrow{j_{-\alpha}} \cH^{-\alpha}(D_e) \cap L^2_K(\Omega;E) \xlongrightarrow{Tf} L^2(\Omega;E) .
$$
The first arrow is compact by \cref{FormalRellich}, and the second is bounded by assumption. We conclude that $Tf$ is a compact operator.
\end{proof}
We note that there is an analogous version of \cref{NegOrderCpt} for operators $T$ with the property that $fT$ extends to a bounded map $L^2(\Omega;E)\to \cH^{\alpha}(D_e)$ for some $\alpha>0$. This remark also applies to the following lemma. 
\begin{lem}\label{CorNegOrder}
Suppose that $T$ is a bounded operator on $L^2(\Omega;E)$ and extends to a bounded map $\cH^{-\alpha}(D_e)\to L^2(\Omega;E)$ for some $\alpha >0$. Then $Tf$ is compact for all $f\in C_\cc^\infty(\Omega^\circ)$. If $D_e$ is semi-regular, then $Tf$ is compact for all $f\in C_c^\infty(\olO)$.
\end{lem}

\begin{proof}
  In either case we may factor $Tf$ as, where $K$ is the support of $f$,
  $$
    L^2(\Omega;E) \xlongrightarrow{f} L^2_K(\Omega;E) \xlongrightarrow{j_{-\alpha}} \cH^{-\alpha}(D_e) \cap L^2_K(\Omega;E) \xlongrightarrow{T} L^2(\Omega;E) ,
  $$
  where the middle map is compact by \cref{FormalRellich}.
\end{proof}
An analogous version of \cref{NegOrderCpt} also holds for commutators, though here it is more convenient to state the condition for all functions simultaneously:
\begin{prop}\label{CommLowerOrder}
Let $D_e$ be a boundary condition $D$. Let $T$ be a bounded operator on $L^2(\Omega;E)$, and let $\alpha>0$. 
\begin{enumerate}
\item\label{CommLowerOrder_a} If the commutators $[f,T]$ and $[f,T^*]$ extend to bounded maps $\cH^{-\alpha}(D_e)\to L^2(\Omega;E)$ for all $f\in C^\infty_{cc}(\Omega^\circ)$, then $[f,T]$ is compact for all such $f$.
\item\label{CommLowerOrder_b} Let $D_e$ be semi-regular. If the commutators $[f,T]$ and $[f,T^*]$ extend to bounded maps $\cH^{-\alpha}(D_e)\to L^2(\Omega;E)$ for all $f\in C^\infty_{c}(\olO)$, then $[f,T]$ is compact for all such~$f$.
\end{enumerate}
\end{prop}
\begin{proof}
We prove \ref{CommLowerOrder_a}, the proof of \ref{CommLowerOrder_b} is analogous. Let $f\in C^\infty_{cc}(\Omega^\circ)$, and choose $\psi \in C^\infty_{cc}(\Omega^\circ)$ with $\psi f=f$. Write
$$
[f,T] = [\psi f, T] = \psi [f,T] + [\psi,T]f = -\left( [f,T^*] \bar{\psi} \right)^*+ [\psi,T]f .
$$
Both $[f,T^*]$ and $[\psi,T]$ extend to bounded operators $\cH^{-\alpha}(D_e) \to L^2(\Omega;E)$ by assumption. The claim then follows from \cref{CorNegOrder}. 
\end{proof}
\subsection{Commutators with functions}\label{SectionFormalPseudolocal}
In this section we investigate commutators of functions of a boundary condition $D_e$ with smooth functions on $\Omega$. The result will be analogous to the fact that the commutator of an operator of order $\beta$ with a smooth function has order $\beta -1$. In particular, we will be able to conclude the pseudo-locality of the bounded transform $D_e(D_e^*D_e+I)^{-\frac{1}{2}}$.

Holomorphic functional calculus is our tool of choice. Suppose that $T$ is a closed operator, $U\subseteq \C$ an open neighborhood of the spectrum $\sigma(T)$, $\psi$ a holomorphic function on $U$, and $\Gamma$ a curve in $U\setminus \sigma(T)$ circumscribing $\sigma(T)$ in positive direction. Then, the holomorphic functional calculus is defined as 
$$
\psi(T) = \frac{i}{2\pi} \int_\Gamma \psi(z) (T-z)^{-1} dz .
$$ 
If $T$ is self-adjoint, then it coincides with the Borel functional calculus. In this case, the holomorphic functional calculus is a convenient way to obtain an integral representation in terms of resolvents, and this is how we use it. Indeed, it essentially suffices to prove the desired pseudo-compactness for resolvents, and the statement for holomorphic functions of the operator then follows via holomorphic functional calculus. The use of integral representations to reduce pseudo-locality to the case of resolvents is a standard technique (see for example \cite{Kasparov1988} or \cite{BDT1989}). We use the holomorphic functional calculus because it covers a reasonably large class of functions\footnote{The holomorphic functional calculus can also be extended to functions with almost-analytic extensions (see for example \cite[Chapter 8]{Dimassi1999}), including for example compactly supported smooth functions, so that our results can be extended to this class of functions. We will not have need for this however.}, but the choice is ultimately arbitrary. 

As usual let $D_e$ be a boundary condition for the elliptic operator $D$ of order $d$. Let $f\in C^\infty_c(\olO)$ be a real-valued function. Throughout this section we assume that either
\begin{itemize}
\item $D_e$ is arbitrary, and $f\in C^\infty_{cc}(\Omega)$, or
\item $D_e$ is local and semi-regular, and $f$ is arbitrary.
\end{itemize}
Throughout, we also let $\mu \coloneqq \sfrac{1}{d}$.
\begin{lem}\label{CommutatorBasicLemma}
  Let  $D_e$ be a boundary condition for $D$ and $f\in C^\infty_c(\olO)$. Suppose either
  \begin{enumerate}
    \item\label{CommutatorBasicLemma_a} $D_e$ is arbitrary, and $f\in C^\infty_{cc}(\Omega)$,
    \item\label{CommutatorBasicLemma_b} or $D_e$ is local and semi-regular, and $f$ is arbitrary. 
  \end{enumerate}
  Then, the commutator $[f,D_e]$ extends to a bounded operator from $\cH^{1-\mu}(D_e)$ to $L^2(\Omega;F)$, and $[f,D_e^*D_e]$ extends to a bounded operator $\cH^{1-\mu}(D_e)\to \cH^{-1}(D_e)$.
\end{lem}
\begin{proof}
Consider case \ref{CommutatorBasicLemma_a}, and let $\psi\in C^\infty_{cc}(\Omega^\circ)$ be such that $\psi=1$ on the support of $f$. Then $[f,D_e]=[f,D_e]\psi$. By \cref{MultInt} and interpolation, $\psi$ yields a continuous map $\cH^{1-\mu}(D_e)\to \cH^{1-\mu}(D_{\min})$. Hence, if $[f,D_{\min}]$ extends to a continuous map $\cH^{1-\mu}(D_{\min})\to L^2(\Omega;E)$, we can factor $[f,D_e]$ as
$$
  \cH^{1-\mu}(D_e) \xlongrightarrow{\psi} \cH^{1-\mu}(D_{\min}) \xlongrightarrow{[f,D_{\min}]} L^2(\Omega;E).
$$
Similarly, if $[f,D^*_{\min}D_{\min}]$ extends to a continuous map $\cH^{1-\mu}(D_{\min})\to \cH^{-1}(D_{\min})$, then $[f,D_e^*D_e]$ factors as 
$$
\cH^{1-\mu}(D_e) \xlongrightarrow{\psi} \cH^{1-\mu}(D_{\min}) \xlongrightarrow{[f,D^*_{\min}D_{\min}]} \cH^{-1}(D_{\min}) \xlongrightarrow{\psi^*} \cH^{-1}(D_e) .
$$
Here, $\psi^*$ is the dual map to multiplication by $\psi$, which acts as the identity on the image of $[f,D^*_{\min}D_{\min}]$. We conclude that case \ref{CommutatorBasicLemma_a} is again a special case of case \ref{CommutatorBasicLemma_b}.

Thus, let $D_e$ be semi-regular and local, and $f$ be arbitrary. \cref{lem_regular_interpolation} provides a continuous inclusion $\cH^{1-\mu}(D_e) \to H^{d-1}_\loc(\Omega;E)$. Since $[f,D]$ is a compactly supported differential operator of order $d-1$, it continuously maps $H^{d-1}_\loc(\Omega;E)$ to $L^2(\Omega;E)$. Therefore, $[f,D_e]$ factors as 
$$
  \cH^1(D_e) \longrightarrow \cH^{1-\mu}(D_e) \longrightarrow H^{d-1}_\loc(\Omega;E) \xlongrightarrow{[f,D_e]} L^2(\Omega;F).
$$
For the second commutator, note that it suffices by \cref{FormalSobolevStructure} to show that
$$
\left| \left\langle [f,D_e^*D_e] u, v \right\rangle_{L^2(\Omega)} \right| \lsim \|u\|_{\cH^{1-\mu}(D_e)} \|v\|_{\cH^1(D_e)}
$$
for all $u,v \in \cH^1(D_e)$. We may assume $f$ to be real-valued. Then we have
$$
\left\langle [f,D_e^*D_e] u, v \right\rangle_{L^2(\Omega;E)} = \left\langle [f,D_e]u,D_ev \right\rangle_{L^2(\Omega;E)} - \left\langle D_eu,[f,D_e]v\right\rangle_{L^2(\Omega;E)} .
$$
Again, due to \cref{lem_regular_interpolation} we have $v,w \in H^d_\loc(\Omega;E)$, which we may assume to be supported inside some precompact neighborhood $V\subseteq \olO$ of $K$. Then,
$$
\left\langle [f,D_e]u,D_ev \right\rangle_{L^2(\Omega;E)} \leq \|[f,D]u\|_{L^2(V;F)} \|Dv\|_{L^2(V;F)} \lsim \|u\|_{H^{d-1}(V;E)} \|v\|_{H^1(V;E)}
$$
as a consequence of the Cauchy--Schwarz inequality and the boundedness of $[f,D]$ and $D$ on the Sobolev scale. Since these Sobolev norms are bounded by the respective graph norms (again due to \cref{lem_regular_interpolation}), this shows the desired norm estimate for the first summand. For the second summand we instead use the fact that the $L^2$-inner product extends to a perfect pairing\footnote{We have eschewed introducing negative-order Sobolev spaces in \cref{sec_sobolev}, so the reader is refered to e.g. \cite[Section~IV.5]{Taylor1996}, specifically Proposition 5.1 therein.} between $H^{-1}(V;F)$ and $H^1_0(V;F)$ to estimate that
$$
\left\langle D_eu,[f,D_e]v\right\rangle_{L^2(\Omega;E)} \lsim  \|Du\|_{H^{-1}(V;F)} \|[f,D]v\|_{H^1(V;F)} \lsim \|u\|_{H^{d-1}(V;E)} \|v\|_{H^1(V;E)} .
$$
Another appeal to \cref{lem_regular_interpolation} gives the desired norm estimate for the second summand as well. This concludes the proof.
\end{proof}
From here on our calculations are purely formal, and dependent only on the validity of the above lemma. We first deduce that the commutator of the resolvent with $f$ has lower order.
\begin{lem}\label{CommutatorResolvent}
Let $z\in \C\setminus [0,\infty)$ and $\gamma\geq 0$. The commutator $\left[ f, \left(D_e^*D_e-z\right)^{-1} \right]$ extends to a bounded operator $\cH^{-1-\mu}(D_e)\to \cH^{1-\gamma}(D_e)$ with norm $\lsim |z|^{-\frac{\gamma}{2}}$.
\end{lem}
\begin{proof}
The formula
\begin{equation}\label{comm_res}
  \left(D_e^*D_e-z\right)\left[ f, \left(D_e^*D_e-z\right)^{-1} \right] \left(D_e^*D_e-z\right) = \left[ D_e^*D_e,f \right]
\end{equation}
is straightforwardly verified to hold on $\cH^{\infty}(D_e)$. It follows that the left-hand side extends to a map $\cH^{1-\mu}(D_e) \to \cH^{-1}(D_e)$, or equivalently that $\left[ f, \left(D_e^*D_e+I\right)^{-1} \right]$ extends to a bounded operator $\cH^{-1-\mu}(D_e)\to \cH^{1}(D_e)$. Thus it also extends to a bounded operator $\cH^{-1-\mu}(D_e)\to \cH^{1-\gamma}(D_e)$ for all $\gamma\geq 0$. 

For the norm estimate, note that the norm of $(D_e^*D_e-z)^{-1}$ as a bounded map $\cH^{\alpha}(D_e)\to \cH^{\alpha+2-\gamma}(D_e)$ is bounded by $\|(x^2-z)^{-1}\|_{\cF^{2-\gamma}}$ by \cref{MappingPropertiesT}. An elementary calculation shows this to be proportional to $|z|^{-\frac{\gamma}{2}}$. Using \eqref{comm_res} and the fact that $\left[ D_e^*D_e,f \right]$ is a bounded map $\cH^{1-\mu}(D_e) \to \cH^{-1}(D_e)$, one deduces that 
$$
\left\|\left[ f, \left(D_e^*D_e-z\right)^{-1} \right]\right\| \lsim |z|^{-\frac{\gamma}{2}}
$$
as a map $\cH^{-1-\mu}(D_e)\to \cH^{1-\gamma}(D_e)$.
\end{proof}
We extend from resolvents to holomorphic functions via the holomorphic functional calculus. Let $\beta\in\R$, and define $\cO^{-\beta}$ as the space of holomorphic functions $\psi \colon U\to \C$ with $U\subseteq \C$ an open neighborhhod of the positive real axis, such that $\sup_{z\in U} |z|^{\beta}|\psi(z)| <\infty$.
\begin{prop}\label{FormalPseudoLocality}
Let $\psi \in \cO^{-\beta}$ for $0<\beta\leq 1$. Then $\left[  f, \psi(D_e^*D_e+I)\right]$ extends to a continuous map $\cH^{-1-\mu}(D_e)\to \cH^{-1+2\beta-\gamma}(D_e)$ for all $\gamma>0$.
\end{prop}
\begin{proof}
According to the holomorphic functional calculus we may write
$$
\psi(D_e^*D_e+I) = \frac{i}{2\pi} \int_\Gamma \psi(z)\cdot (D_e^*D_e+1-z)^{-1} \, dz ,
$$
so that 
$$
\left[f,\psi(D_e^*D_e+I)\right] = \frac{i}{2\pi} \int_\Gamma \psi(z)\cdot \left[f,(D_e^*D_e+1-z)^{-1}\right] \, dz.
$$
By \cref{CommutatorResolvent}, $\left[ f, \left(D_e^*D_e-z\right)^{-1} \right]$ extends to a bounded map $\cH^{-1-\mu}(D_e)\to \cH^{1-\gamma'}(D_e)$ for all $\gamma'\geq 0$, and its norm is bounded by $|z|^{-\frac{\gamma'}{2}}$. It follows that $\left[  f, \psi(D_e^*D_e+I)\right]$ extends to a bounded operator $\cH^{-1-\mu}(D_e)\to \cH^{1-\gamma'}(D_e)$ whenever the integral $\int_\Gamma \left|\psi(z)\right|\cdot |z|^{-\frac{\gamma'}{2}}\, dz$ converges. Since $\psi \in \cO^{-\beta}$, meaning that $|\psi(z)|\lsim |z|^{-\beta}$, this is the case if $-\beta-\frac{\gamma'}{2}<-1$, i.e. if $\gamma'>2(1-\beta)$. Defining $\gamma \coloneqq 2(\beta-1)+\gamma'$, the claim follows.
\end{proof}
We cannot extend the result of \cref{FormalPseudoLocality} to cover $\gamma=0$ as well. This is because an infinitessimal amount of $\gamma$ is needed to ensure that the integral representation of $[f,\psi(D_e^*D_e+I)]$ converges in the operator norm of maps from $\cH^{-1-\mu}\to \cH^{-1+2\beta-\gamma}$. This is not a fundamental limitation, but rather an artifact of our choice to treat $\psi(D_e^*D_e+I)$ via an integral representation. However, using duality and interpolation we can extend the statement to more domains. Indeed, suppose that $\psi\in\cO^{-\beta}$. It holds that 
$$
\left\langle \left[f, \psi(D_e^*D_e+I)\right]u,v \right\rangle_H = \left\langle u,\left[\overline{\psi}(D_e^*D_e+I),f\right]v \right\rangle_H 
$$
for all $u,v \in \cH^\infty(D_e)$. It (almost) never holds that $\overline{\psi}$ is in $\cO^{-\beta}$, but if $\psi$ is real-valued on $\R_{\geq 0}$, then $\overline{\psi}(D_e^*D_e+I)=\psi(D_e^*D_e+I)$. In this case the conditions for \cref{DualityInterpolation} are satisfied, and \cref{FormalPseudoLocality} can be strengthened as follows.
\begin{cor}\label{FormalPseudoLocStrengthened}
Let $\psi \in \cO^{-\beta}$ with $\beta\geq 0$, and suppose that $\psi$ is real-valued on $\R_{\geq 0}$. Then $\left[  f, \psi(D_e^*D_e+I)\right]$ extends to a bounded operator $\cH^{\alpha}(D_e)\to \cH^{\alpha+2\beta+\mu-\gamma}(D_e)$ for every $\alpha\in[-1-\mu,1-2\beta]$ and all $\gamma>0$.
\end{cor}
The additional range provided by the previous corollary allows us to prove that the commutator of the bounded transform with $f$ has negative order.
\begin{cor}\label{BoundedTransLowerOrder}
The commutator $\left[f, D_e(D_e^*D_e+I)^{-\frac{1}{2}}\right]$ extends to a bounded operator $\cH^{-\gamma}(D_e)\to L^2(\Omega;E)$ for all $0<\gamma<\mu$.
\end{cor}
\begin{proof}
Write
$$
\left[f, D_e(D_e^*D_e+I)^{-\frac{1}{2}}\right] = \left[ f,D_e\right](D_e^*D_e+I)^{-\frac{1}{2}} + D_e \left[ f, (D_e^*D_e+I)^{-\frac{1}{2}} \right] .
$$
The first summand extends to a bounded operator $\cH^{-\mu}(D_e)\to L^2(\Omega;E)$. For the second summand note that $(z^2+1)^{-\frac{1}{2}} \in \cO^{-\frac{1}{2}}$ is real-valued for all $z\in\R_{\geq 0}$. The commutator $\left[ f, (D_e^*D_e+I)^{-\frac{1}{2}} \right]$ thus extends to a bounded operator $\cH^{-\gamma}(D_e)\to \cH^{-\gamma+1+\mu-\gamma'}(D_e)$ for all $0\leq\gamma\leq 1+\mu$ and all $\gamma'>0$ by the previous corollary. Taking $0<\gamma'<\gamma<\mu$ we get that $\left[ f, (D_e^*D_e+I)^{-\frac{1}{2}} \right]$ extends to a bounded operator $\cH^{-\gamma}(D_e) \to \cH^1(D_e)$. Since $D_e$ maps $\cH^1(D_e)$ to $L^2(\Omega;E)$ the claim follows.
\end{proof}
Appealing to \cref{CommLowerOrder} lets us now conclude the desired pseudo-locality statements:
\begin{cor}\label{BdTComm}
Let $D_e$ be a boundary condition for the elliptic operator $D$ over $\olO$.
\begin{enumerate}
\item $\left[ f,D_e(D_e^*D_e+I)^{-\frac{1}{2}} \right]$ is compact for all functions $f\in C_\cc^\infty(\Omega^\circ)$.
\item If $D_e$ is local and semi-regular, then $\left[ f,D_e(D_e^*D_e+I)^{-\frac{1}{2}} \right]$ is compact for all functions $f\in C_c^\infty(\olO)$.
\end{enumerate}
\end{cor}
\subsection{Lower-order perturbations}\label{SectionLowerOrderPerturbations}
We come to the comparison of boundary conditions for two elliptic operators that differ only by lower-order terms. In particular this includes the case of two boundary conditions for the same operator. Our strategy is analogous to that of the previous section: Use comparison to Sobolev spaces to argue that the difference between the operators has lower order on the formal scale, deduce that the analogous statement holds for resolvents as well, then extend to holomorphic functions via holomorphic functional calculus, and finally derive the desired result for the bounded transforms. This process is slightly complicated by the fact that the formal scales of the two boundary conditions are generally different, so that there is no common scale on which to work. This issue is solved by interior regularity. Thus the results of this section work only on the interior $\Omega^\circ$. While an extension to $\olO$ might be possible in favorable cases, we will not investigate this further. 

Throughout this section we consider the following situation: 
Let $D$ and $\tilde{D}$ be elliptic operators over $\olO$ acting on sections of the same bundles. Assume that $D$ and $\tilde{D}$ both have order $d$, and that $D-\tilde{D}$ is a differential operator of order $\leq d-1$. Let $D_e$ and $\tilde{D}_e$ be boundary conditions for $D$ and $\tilde{D}$ respectively. Note that $D_e$ and $\tilde{D}_e$ could in particular be two boundary conditions for the same operator. We set $\mu \coloneqq \sfrac{1}{d}$. Lastly, let $f\in C^\infty_{cc}(\Omega)$ be a smooth function with compact support in the interior.

Let us start with the following observation. Interior regularity for the operators $D$ and $\tilde{D}$ implies that $f$ maps the domains of $D_e$ and $\tilde{D}_e$ to $\tilde{H}^d(V;E) \subseteq \dom(D_e)\cap \dom(\tilde{D}_e)$ for $V\subseteq\Omega^\circ$ a precompact Lipschitz domain containing the support of $f$. Similarly, $f$ maps the domains of $D^*_e D_e$ and $\tilde{D}^*_e \tilde{D}_e$ to $\tilde{H}^{2d}(V;E) \subseteq \dom(D^*_e D_e)\cap\dom(\tilde{D}^*_e \tilde{D}_e)$. As a consequence, it makes sense to consider $(D_e^*D_e - \tilde{D}_e^*\tilde{D}_e)f$ as a bounded operator from $\cH^2(D_e)=\dom(D_e^*D_e)\to L^2(\Omega;E)$. Then, by arguing via Sobolev spaces in an analogous manner to the last section one proves the following:
\begin{lem}\label{DifferencesBasicLemma}
The operators $[f,D_e^*D_e]$ and $\left( D_e^*D_e - \tilde{D}_e^*\tilde{D}_e \right)f$ extend to bounded operators $\cH^{2-\mu}(D_e)\to L^2(\Omega;E)$.
\end{lem} 
From this we may again derive the corresponding statement for resolvents. Now we multiply the difference of the resolvents from both sides by functions with compact support in the interior. This produces an operator that acts on a single scale. 
\begin{lem}
Let $f,g\in C^\infty_{cc}(\Omega^\circ)$, and let $\gamma\geq 0$, $z\in\C\setminus[0,\infty)$. The operator 
$$
g\left( (D_e^*D_e-z)^{-1} - (\tilde{D}^*_e\tilde{D}_e-z)^{-1} \right)f
$$ 
extends to bounded operators $\cH^{-1-\mu}(D_e)\to \cH^{1-\gamma}(D_e)$, with norm $\lsim |z|^{-\frac{\gamma}{2}}$.
\end{lem}
\begin{proof}
Write $R \coloneqq D_e^*D_e-z$ and $\tilde{R} \coloneqq \tilde{D}_e^*\tilde{D}_e-z$. On the intersection $\dom(R)\cap \dom(\tilde{R})$ the formula
$$
\left( R^{-1}- \tilde{R}^{-1} \right) R = \tilde{R}^{-1} \left( \tilde{R}- R \right)
$$
holds. Since $f$ maps $\dom(R)$ to $\dom(R)\cap \dom(\tilde{R})$ by \cref{MinMaxDomain}, we can thus calculate as follows:
\begin{align*}
\left( R^{-1} - \tilde{R}^{-1} \right)f &= \left( R^{-1} - \tilde{R}^{-1} \right) \left( RfR^{-1} + \left[f,R\right]R^{-1} \right) \\
&= \tilde{R}^{-1}\left(\left( \tilde{R}-R \right) f - \left[f,R\right]\right)R^{-1} + R^{-1}\left[f,R\right]R^{-1} .
\end{align*}
By the previous lemma the operator $\left(\left( \tilde{R}-R \right) f - \left[f,R\right]\right)R^{-1}$ extends to a bounded operator $\cH^{-\mu}(D_e)\to L^2(\Omega;E)$, so that the first summand extends to a bounded operator $\cH^{-\mu}(D_e)\to \cH^{2-\gamma}(\tilde{D}_e)$ with norm $\lsim |z|^{-\frac{\gamma}{2}}$ (the norm bound is derived analogously to that in \cref{CommutatorResolvent}). The second summand extends to a bounded operator $\cH^{-\mu}(D_e)\to \cH^{2-\gamma}(D_e)$, again with an analogous norm bound. Now, multiplication by $g$ is a bounded map on $\cH^{2-\gamma}(D_e)$. We conclude that $g(R^{-1}-\tilde{R}^{-1})f$ extends to a bounded operator $\cH^{-\mu}(D_e)\to \cH^{2-\gamma}(D_e)$ with norm $\lsim |z|^{-\frac{\gamma}{2}}$. Note that the same is also true for 
$$
\left( g\left(R^{-1}-\tilde{R}^{-1}\right)f \right)^* = \bar{f}\left(R^{-1}-\tilde{R}^{-1}\right)\bar{g} .
$$
Thus \cref{DualityInterpolation} implies that $g(R^{-1}-\tilde{R}^{-1})f$ actually extends to a bounded map $\cH^{\alpha}(D_e)\to \cH^{\alpha+\mu+2-\gamma}(D_e)$ for all $\alpha$ between $-2$ and $-\mu$ with an analogous norm bound. The claim is the special case of $\alpha=-1-\mu$.
\end{proof}
Analogously to \cref{FormalPseudoLocality} and \cref{FormalPseudoLocStrengthened} the result may again be extended to functions of $D_e^*D_e$ and $\tilde{D}_e^*\tilde{D}_e$ via holomorphic functional calculus. 
\begin{prop}\label{LowerOrderPsi}
Let $f,g\in C^\infty_{cc}(\Omega^\circ)$, and let $\psi \in \cO^{-\beta}$, $0\leq\beta\leq 1$, be real-valued on $\R_{\geq 0}$. Then, the operator 
$$
g\left( \psi(D_e^*D_e+I) - \psi(\tilde{D}_e^*\tilde{D}_e+I) \right)f
$$
extends to bounded maps $\cH^{\alpha}(D_e) \to \cH^{\alpha+\mu+2\beta-\gamma}(D_e)$ and $\cH^{\alpha}(\tilde{D}_e) \to \cH^{\alpha+\mu+2\beta-\gamma}(\tilde{D}_e)$ for all $\alpha\in[-1-\mu,1-2\beta]$ and all $\gamma>0$.
\end{prop}
Again, we use this to get the local compactness statement for the bounded-transforms that we are ultimately after:
\begin{cor}\label{BdTLowerOrderPerturbation}
  The operator
  $$
    \left( D_e\left(D_e^*D_e+I\right)^{-\frac{1}{2}} - \tilde{D}_e\left(\tilde{D}_e^*\tilde{D}_e + I\right)^{-\frac{1}{2}} \right)f
  $$
is compact for all $f\in C_\cc^\infty(\Omega^\circ)$.
\end{cor}
\begin{proof}
Fix $f\in C^\infty_\cc(\Omega^\circ)$, together with a function $\psi\in C_\cc^\infty(\Omega^\circ)$ such that $\psi=1$ on the support of $f$. Then, write
\begin{align*}
&\left( D_e\left(D_e^*D_e+I\right)^{-\frac{1}{2}} - \tilde{D}_e\left(\tilde{D}_e^*\tilde{D}_e + I\right)^{-\frac{1}{2}} \right)f \\
&= D_e f \left(D_e^*D_e+I\right)^{-\frac{1}{2}} - \tilde{D}_e f \left(\tilde{D}_e^*\tilde{D}_e + I\right)^{-\frac{1}{2}}\psi \\
&\quad - D_e\left[f,\left(D_e^*D_e+I\right)^{-\frac{1}{2}}\right] + \tilde{D}_e \left[ f,\left(\tilde{D}_e^*\tilde{D}_e + I\right)^{-\frac{1}{2}} \right]\psi \\
&= \left( D_e - \tilde{D}_e\right) f \left(D_e^*D_e+I\right)^{-\frac{1}{2}} + \tilde{D}_e f \left( \left(D_e^*D_e+I\right)^{-\frac{1}{2}} - \left(\tilde{D}_e^*\tilde{D}_e + I\right)^{-\frac{1}{2}} \right)\psi \\
&\quad - D_e\left[f,\left(D_e^*D_e+I\right)^{-\frac{1}{2}}\right] + \tilde{D}_e \left[ f,\left(\tilde{D}_e^*\tilde{D}_e + I\right)^{-\frac{1}{2}} \right]\psi
\end{align*}
The first summand extends to a bounded map $\cH^{-\mu}(D_e)\to L^2(\Omega;E)$, because $(D_e-\tilde{D}_e)f$ extends to a bounded map $\cH^{1-\mu}(D_e)\to L^2(\Omega;E)$. The third summand extends to a bounded map $\cH^{-\mu}(D_e)\to L^2(\Omega;E)$ by the results of the previous section. The fourth summand analogously extends to a bounded operator $\cH^{-\mu}(\tilde{D}_e)\to L^2(\Omega;E)$. It remains to treat the second summand. \cref{LowerOrderPsi} implies via an analogous argument to that in the proof of \cref{BoundedTransLowerOrder} that 
$$
f \left( \left(D_e^*D_e+I\right)^{-\frac{1}{2}} - \left(\tilde{D}_e^*\tilde{D}_e + I\right)^{-\frac{1}{2}} \right)\psi
$$
extends to a bounded operator $\cH^{-\gamma}(\tilde{D}_e)\to \cH^1(\tilde{D}_e)$ for $0<\gamma<\mu$. Thus the second summand extends to a bounded map $\cH^{-\gamma}(\tilde{D}_e)\to L^2(\Omega;E)$. We conclude that 
$$
    \left( D_e\left(D_e^*D_e+I\right)^{-\frac{1}{2}} - \tilde{D}_e\left(\tilde{D}_e^*\tilde{D}_e + I\right)^{-\frac{1}{2}} \right)f
$$
extends to a bounded operator $\cH^{-\gamma}(\tilde{D}_e)\to \cH^1(\tilde{D}_e)$ for each $0<\gamma<\mu$, so that the claim follows from \cref{NegOrderCpt}. 
\end{proof}

\section{\texorpdfstring{$\boldsymbol{K}$}{K}-homology classes from boundary conditions}\label{SecMainResult}

In this section, we prove our main theorem on the $K$-homology classes defined by boundary conditions (\cref{ExistencKHomClasses}). Afterwards, we will discuss examples, connect our set-up to the slightly different one used by Baum--Douglas--Taylor in the graded case, and briefly discuss finite-propagation results and coercive boundary conditions. Lastly, in \cref{SecNonlocal} we weaken the locality required for our main theorem, and discuss classes defined by non-local boundary conditions. 

\subsection{The main theorem for local boundary conditions}

  We briefly recall the definition of Fredholm modules, the cycles for analytic $K$-homology. A detailed exposition on Fredholm modules and $K$-homology can be found in \cite{HigsonRoe2000}. We also reference the multigrading terminology used therein.
  \begin{defn}
    Let $X$ be a locally compact Hausdorff space. An \emph{ungraded Fredholm module} over $X$ is a triple $(H,\rho,F)$ consisting of a Hilbert space $H$, a representation $\rho\colon\, C_0(X)\to\fBH$, and a bounded operator $F\in\fBH$, so that:
    \begin{enumerate}
      \item $F$ is pseudo-local, i.e. $[\rho(f),F]$ is compact for all $f\in C_0(X)$,
      \item\label{item_square} $I-F^2$ is locally compact, i.e. $(I-F^2)\rho(f)$ is compact for all $f\in C_0(X)$,
      \item\label{item_adjoint} $F-F^*$ is locally compact, i.e. $(F-F^*)\rho(f)$ is compact for all $f\in C_0(X)$.
    \end{enumerate}
    If $H$ is graded, the representation $\rho$ is by even operators and $F$ is an odd operator, the triple is called a \emph{graded Fredholm module}. If $H$ is in addition $p$-multigraded ($p>0$) and the operators $F$ and $\rho(f)$ commute with the multigrading operators, then the triple is called a \emph{$p$-multigraded Fredholm module}. Lastly, if $Z\subseteq X$ is a closed subspace, then an \emph{ungraded/graded/multigraded relative Fredholm module} is defined as an analogous triple, where \cref{item_square} and \cref{item_adjoint} are only required to hold for functions in $C_0(X\setminus Z)$.
  \end{defn}
  Conventionally, a graded Fredholm module is called $0$-multigraded, and an ungraded Fredholm module $(-1)$-multigraded. Then, for $p\geq -1$ the $(-p)$th analytic $K$-homology group $K_{-p}(X)$ is defined as the free group generated by unitary equivalence classes of $p$-multigraded Fredholm modules, modulo homotopies and direct sums. The relative $K$-homology groups $K_{-p}(X,Z)$ are defined analogously using relative Fredholm modules. 

With the technical prerequisites established in \cref{SectionOpCalc}, proving our main theorem is largely a matter of assembling what was already shown.

\begin{thm}\label{ExistencKHomClasses}
Let $M$ be a smooth Riemannian manifold (possibly with boundary or incomplete) and $\Omega \subset M$ be a Lipschitz domain. Denote by $\Omega^\circ:=\Omega\setminus\partial M$ and $\dO:=\olO\setminus\Omega^\circ$ the interior and boundary of $\Omega$ respectively. 

Let $E\to \Omega$ be a $p$-multigraded vector bundle for $p\geq -1$. Let $D\colon \Gamma_{cc}(\Omega;E)\to \Gamma_{cc}(\Omega;E)$ be a symmetric differential operator which is elliptic over $\olO$. Let $D_e$ be a boundary condition for $D$. If $p\geq 0$, assume that both $D$ and $D_e$ are odd and $p$-multigraded. Set $H \coloneqq L^2(\Omega;E)$ and let $\rho$ be the representation of $C_0(\olO)$ on $H$ by multiplication operators. We denote the restriction of $\rho$ to $C_0(\Omega^\circ)$ by $\rho$ as well. Lastly, set 
$$
F_e \coloneqq D_e(D_e^*D_e+I)^{-\frac{1}{2}} .
$$
Then:
\begin{itemize}
\item[i)]\label{mainthm_local_item_i} $(H,\rho,F_e)$ is a $p$-multigraded Fredholm module over $\Omega^\circ \coloneqq \Omega \setminus \partial M$.
\item[ii)]\label{mainthm_local_item_ii} Assume that $D_e$ is local and semi-regular. Then $(H,\rho,F_e)$ is a $p$-multigraded relative Fredholm module over $(\olO,\partial \Omega)$.
\item[iii)]\label{mainthm_local_item_iii} Assume that $D_e$ is local, self-adjoint and regular. Then $(H,\rho,F_e)$ is a $p$-multigraded Fredholm module over $\olO$.
\end{itemize}
The resulting K-homology classes enjoy the following properties:
\begin{itemize}
\item[a)] Assume that $D_e$ is local and semi-regular. The excision map $\exc\colon K_{-p}(\olO,\partial \Omega) \to K_{-p}(\Omega^\circ)$ maps the class defined by ii) to the class defined by i).
\item[b)] Assume that $D_e$ is local, self-adjoint and regular. The restriction map $K_{-p}(\olO) \to K_{-p}(\olO,\partial \Omega)$ maps the class defined by iii) to the class defined by ii).
\item[c)] The classes in $K_{-p}(\Omega^\circ)$ and $K_{-p}(\olO;\dO)$ defined by i) and ii), respectively, do not dependent on the boundary condition and depend only on the principal symbol of~$D$.
\item[d)]\label{ItemNormFun} Assume that the boundary condition $D_e$ is self-adjoint. Then, the conclusions of i), ii) and iii) also hold with $D_e(D_e^*D_e+I)^{-\frac{1}{2}}$ replaced by $\chi(D_e)$ for any continuous odd function $\chi\colon \R\to\R$ with $\lim_{\lambda\to\infty}\chi(\lambda)=1$. The resulting K-homology classes are independent of $\chi$.
\end{itemize}
\end{thm}
\begin{proof}
The properties necessary to establish i), ii) and iii) were essentially proved over the course of \cref{SectionOpCalc}.

Let us begin with i). That $[F_e,f]$ is compact for $f\in C_0(\Omega^\circ)$ follows from \cref{BdTComm}. Furthermore, $I- F_e^*F_e  = (D_e^*D_e+I)^{-1}$ extends to a bounded map $\cH^{-2}(D_e)\to L^2(\Omega;E)$, so that \cref{CorNegOrder} yields that $(I-F_e^*F_e)f$ is compact for $f\in C_0(\Omega^\circ)$.

Suppose that $D_d$ is another boundary condition for $D$, and let $F_d \coloneqq D_d(D_d^*D_d+I)^{-1/2}$. It follows from \cref{BdTLowerOrderPerturbation} that $(F_e-F_d)f$ is compact for $f\in C_0(\Omega^\circ)$. Since $D$ is assumed symmetric, $D_d \coloneqq D_e^*$ is also a boundary condition for $D$, and this particular case yields that $(F_e-F_e^*)f$ is compact for $f\in C_0(\Omega^\circ)$.

Lastly, if $p\geq 0$ and $D_e$ is odd and $p$-multigraded, then $D^*_eD_e$ is even and $p$-multigraded, hence so is $(D_e^*D_e+I)^{-\frac{1}{2}}$. It follows that $F_e$ is odd and $p$-multigraded. This proves i).

In the above argument, we have already seen that if $F_d$ is the bounded transform of a different extension $D_d$, then $(F_e-F_d)f$ is compact for $f\in C_0(\Omega^\circ)$. Therefore $F_e$ and $F_d$ differ by a compact perturbation, and thus define the same class in $K_*(\Omega^\circ)$. More generally, \cref{BdTLowerOrderPerturbation} yields that if $D$ and $\tilde{D}$ are two elliptic operators of order $d$ over $\olO$ that differ by an operator of order $\leq d-1$, then the Fredholm modules defined by any boundary conditions for $D$ and $\tilde{D}$ differ by a compact perturbation over $\Omega$, proving c) in the case of absolute classes over the interior $\Omega^\circ$.

For ii) it remains to show that $[f,F_e]$ is compact for $f\in C_0(\olO)$ if $D_e$ is local and semi-regular. This was shown in \cref{BdTComm}. Note that a) then holds since the excision map is the identity at the level of Fredholm modules. If $\tilde{D}_d$ is a local and semi-regular boundary condition of another elliptic operator with the same principal symbol (or in particular, another such boundary condition for $D$), \cref{BdTLowerOrderPerturbation} again gives that $F_e$ and $\tilde{F}_d$ differ by a compact perturbation. Hence the relative case of c) follows.

For iii) we have already noted in ii) that $[f,F_e]$ is compact for $f\in C_0(\olO)$, and that $(F_e-F_e^*)f$ is compact holds trivially due to the assumed self-adjointness. It remains to show that $(I-F_e^*F_e)f$ is compact for $f\in C_0(\olO)$. Because $I-F_e^*F_e=(D_e^*D_e+I)^{-1}$ extends to a bounded map $\cH^{-2}(D_e)\to L^2(\Omega;E)$, this is again implied by \cref{CorNegOrder} via (semi-)regularity. This concludes the proof of iii). Then b) holds because the restriction map $K_*(\olO)\to K_*(\olO,\dO)$ is the identity at the level of Fredholm modules.

It remains to prove d). Note that the operator $F_e$ considered thus far corresponds to the normalizing function $\chi_0(x)=x(x^2+1)^{-\frac{1}{2}}$. By a density argument it suffices to consider normalizing functions $\chi$ that differ from $\chi_0$ by a function in $\cF^{-\beta}$ for some $\beta>0$. Then, $(\chi-\chi_0)(D_e)$ extends to a map $\cH^{-\beta}(D_e)\to L^2(\Omega;E)$ by \cref{MappingPropertiesT}. If $D_e$ is arbitrary, we conclude from \cref{CorNegOrder} that $\chi(D_e)=F_e+(\chi-\chi_0)(D_e)$ is pseudo-local over $\Omega$ and differs from $F_e$ by a compact perturbation. If $D_e$ is local and semi-regular, we conclude from the same that $\chi(D_e)$ is pseudo-local over $\olO$, and that it differs from $F_e$ by a compact perturbation over $\olO$. This proves d), completing the proof of the theorem.
\end{proof}

\begin{rem}
  As was already the case for our discussion in \cref{sec_sobolev}, the Lipschitz assumption on the boundary of $\olO$ is a matter of technical convenience, not fundamental limitation. In \cite{LykoThesis} the counterpart of \cref{ExistencKHomClasses} in uniform K-homology is established for general open domains. This requires a slightly stronger notion of semi-regularity, roughly speaking imposing regularity in intermediate degrees between $0$ and the order of the operator. If the domain is at least Lipschitz, any semi-regular boundary condition automatically also satisfies this stronger property. Moreover, the minimal boundary condition always possesses this stronger regularity property. In particular, \cref{ex_khom_min} below can be extended to arbitrary open domains $\Omega$. For the purposes of this article, we decided to avoid these technical complications, and give our results under the convenient Lipschitz assumption.
\end{rem}

Let us consider some of the examples of \cref{SectionExamples}.

\begin{ex}[Relative classes and absolute classes over the interior]\label{ex_khom_min}
  The minimal boundary condition exists for every elliptic operator, and is always local, semi-regular, and possesses the appropriate grading properties (see \cref{ex_minimal_condition}). Thus, Items i) and ii) of \cref{ExistencKHomClasses} apply to it. In conjunction with c), we conclude that there are well-defined K-homology classes
  $$
    [D] \in K_{-p}(\Omega^\circ) \quad \text{and} \quad [D] \in K_{-p}(\olO,\dO)
  $$
  that are independent of boundary condition chosen to represent them. The existence of these classes has been established previously in special cases, and owing to the excision isomorphism one does not necessarily have to appeal to boundary conditions to define them. Nevertheless, the explicit representatives obtained from boundary conditions could provide a useful tool in applications.
\end{ex}

\begin{ex}[Absolute and relative boundary conditions]\label{ex_euler_char}
Both the absolute and the relative boundary condition for the Euler characteristic operator are local, self-adjoint and regular, and thus define Fredholm modules over $\olO$. If $\olO$ is compact, the Fredholm index of these operators equals the absolute resp. relative Euler characteristic of~$M$ (see e.g. \cite[Section~4.1]{Gilkey1996}). In particular, they are generally unequal. Since the index can be read off $[D^\chi_{\abs}]$ and $[D^\chi_{\rel}]$ by mapping to a point via $K_0(\olO) \to K_0(\pt) \cong \IZ$, we conclude that in general $[D^\chi_{\abs}] \not= [D^\chi_{\rel}] \in K_0(\olO)$. Hence, the absolute classes over $\olO$ generally depend on the choice of boundary condition.

For a concrete example, we first note that the long exact sequence in homology implies $\chi_\abs(\olO) = \chi_\rel(\olO) + \chi(\dO)$. Thus, every compact manifold $\olO$ with $\chi(\dO) \not= 0$ will do the job, for example a closed ball of odd-dimension. In fact, $\chi(\dO) = (1+(-1)^n)\chi_\abs(\olO)$ for any compact $(n+1)$-dimensional manifold $\olO$ with boundary $\dO$ \cite[Proposition~18.6.2]{tom_dieck}, so that
\[
  \mathrm{index}(D^\chi_\abs) = \chi_\abs(\olO) = \frac{1}{2}\chi(\dO) \quad \text{and} \quad \mathrm{index}(D^\chi_\rel) = \chi_\rel(\olO) =  -\frac{1}{2}\chi(\dO) .\qedhere
\]
\end{ex}

\begin{ex}[APS- and chirality boundary conditions]\label{ExAPSChirality}
The APS-boundary conditions are self-adjoint (if $\ker(\Dirac_{\dO})=0$) and regular, but not local. Thus they define a Fredholm module over $\Omega^\circ$, but not an absolute one over $\olO$. The chirality boundary conditions are local and regular, but not self-adjoint. In even dimensions they further lack the correct grading properties, but in odd dimensions this is no problem. In this case they define an ungraded relative Fredholm module over $(\olO,\dO)$, but due to lack of self-adjointness we cannot conclude from Theorem \ref{ExistencKHomClasses}.iii) that they define a Fredholm module over $\olO$ in this case. 

In fact, by boundary-of-Dirac-is-Dirac the class $[\Dirac]\in K_*(\olO,\dO)$ resulting from the above boundary condtions is mapped to $[\Dirac_{\dO}] \neq 0$ under the K-homological boundary map, so no boundary condition for $\Dirac$ can define an absolute Fredholm module over $\olO$. Hence, \cref{ExistencKHomClasses}.iii) cannot hold without at least some condition on the locality and adjoint of the boundary condition.
\end{ex}

\begin{ex}[Chirality boundary conditions on odd-dimensional manifolds]\label{ex_khom_chiral}
  We just noted in \cref{ExAPSChirality} that the chiral boundary conditions on an odd-dimensional spin manifold do not define an absolute cycle for $K_1(\olO)$. However, the boundary condition
$$
  \hat{\Dirac} = \begin{pmatrix}
    0 & \Dirac_- \\ \Dirac_+ & 0
  \end{pmatrix}
$$
does satisfy all conditions of \cref{ExistencKHomClasses}.iii), hence defines a graded absolute cycle for $K_0(\olO)$. We note however that in the compact case, $\hat{\Dirac}$ has trivial (graded) Fredholm index \cite[Section~8.5]{BaerBallmann2012}. Hence, any potential non-triviality of $[\hat{\Dirac}]\in K_0(\olO)$ would have to be observed by other means. 
\end{ex}

\begin{ex}[Dirichlet Laplacian]\label{DirichletTrivial}
The Dirichlet Laplacian is self-adjoint, local and regular (see \cref{ExNeumann}) and thus $\left( L^2(\olO),\rho, \Delta_\mathrm{Dir}(\Delta_\mathrm{Dir}^2+I)^{-1/2} \right)$ defines an ungraded Fredholm module over $\olO$. Unfortunately, its K-homology class is trivial. To see this, let $F=\Delta_\mathrm{Dir}(\Delta_\mathrm{Dir}^2+I)^{-1/2}$ and note that $F$ is positive. Then, we see that
$$
F_t = \left( \sin(t) + \cos(t)F \right)\left( I + \cos(t)\sin(t) F \right)^{-\frac{1}{2}}
$$
is an operator-homotopy from $F$ to the degenerate Fredholm module $(L^2(\olO),\rho,I)$. More generally suppose that $D$ is any positive symmetric elliptic operator, and $D_e \geq 0$ a positive self-adjoint boundary condition. Then, the class $[D_e] \in K_1(\Omega^\circ)$ vanishes. If $D_e$ is also local and regular, then the class $[D_e]\in K_1(\olO)$ vanishes as well.
\end{ex}

In their work \cite{BDT1989} on K-homology classes for elliptic operators on manifolds with boundary, Baum-- Douglas--Taylor approach the case of relative classes slightly differently. They consider graded odd operators $D=\begin{pmatrix}
0 & D^- \\ D^+ & 0
\end{pmatrix}$, and instead of looking at semi-regular boundary condition for $D$ directly---which is what we did above---they consider semi-regular boundary conditions $D_e^+$ for $D^+$, and form the boundary condition
$$
D_e = \begin{pmatrix}
 0 & (D^+_e)^* \\ D_e^+ & 0
\end{pmatrix} .
$$
It is not semi-regular unless $D^+_e$ is already regular. They then show that the boundary condition $D_e$ defines a relative Fredholm module. We can do the same in our setting:

\begin{prop}\label{KHomGradedComp}
Let $D$ be a graded, odd elliptic operator over $\olO$, and let $D^\pm$ denote the graded components of $D$. Let $D_e^+$ be a local and semi-regular boundary condition for $D^+$, and set 
$$
D_e = \begin{pmatrix}
 0 & (D^+_e)^* \\ D_e^+ & 0
\end{pmatrix} .
$$
Then, $\left( L^2(\Omega;E),\rho,\chi(D_e)\right)$ is a graded relative Fredholm module over $(\olO,\partial \Omega)$. Its class in $K_0(\olO,\partial\Omega)$ coincides with that defined in \cref{ExistencKHomClasses}.ii). In particular, it is independent of the choice of normalizing function.
\end{prop}

We need a lemma.

\begin{lem}\label{SemiRegGradedComp}
Let $D$ and $D_e$ be as above. Then, for any odd function $C_0(\R)$ with $\varphi(0)=0$ it holds that 
$$
[f,\varphi(D_e)] \quad \text{and} \quad \varphi(D_e)g
$$ 
are compact for all $f\in C_0(\olO)$ and $g\in C_0(\Omega^\circ)$.
\end{lem}

\begin{proof}
By way of a density argument it suffices to consider the case where $\varphi(x)=x(x^2+\lambda)^{-\beta}$ for $\lambda >0$ and $\beta>\frac{1}{2}$. In this case
$$
\varphi(D_e) =  \begin{pmatrix}
 0 & \left( D_e^+((D_e^+)^*D^+_e+\lambda)^{-\beta}\right)^* \\  D_e^+((D_e^+)^*D^+_e+\lambda)^{-\beta}
\end{pmatrix} ,
$$
so that it suffices to prove that $[f,D_e^+((D_e^+)^*D^+_e+\lambda)^{-\beta}]$ is compact for all $f$. To see this, it suffices to note that $(x^2+\lambda)^{-\beta} \in \cF^{-2\beta}$, so that $D_e^+((D_e^+)^*D^+_e+\lambda)^{-\beta}$ extends to a bounded operator $\cH^{-2\beta+1}(D_e^+)\to \cH^0((D_e^+)^*)=L^2(\Omega,E^-)\subseteq L^2(\Omega;E)$. Then, \cref{CommLowerOrder} and \cref{CorNegOrder} yield the claim.
\end{proof}

\begin{proof}[Proof of \cref{KHomGradedComp}]
Since $D_e$ is an odd boundary condition, $\chi(D_e)$ defines a graded Fredholm module over $\Omega^\circ$ by \cref{ExistencKHomClasses}.i). Thus it remains to show that $[f,\chi(D_e)]$ is actually compact for all $f\in C_0(\olO)$. First, consider $\chi_0(x)=x(x^2+1)^{-1/2}$. Then, $\chi_0(D_e) = \begin{pmatrix}
0 & F^*\\ F& 0
\end{pmatrix}$ with $F=D_e^+((D_e^+)^*D^+_e+I)^{-1/2}$. Since $D^+_e$ is local and semi-regular, the commutator of $F$ and hence also of $F^*$ with any $f\in C_0(\olO)$ is compact. Thus, $\chi_0(D_e)$ actually defines a relative Fredholm module over $(\olO,\dO)$. Now, an arbitrary normalizing function $\chi$ differs from $\chi_0$ by an odd function $\varphi \in C_0(\R)$ with $\varphi(0) = 0$ vanishing both at infinity and at $0$. Thus, we deduce through \cref{SemiRegGradedComp} that any $\chi(D_e)$ is a compact perturbation of $\chi_0(D_e)$, hence also defines a relative Fredholm module over $(\olO,\dO)$. That it defines the same relative K-homology class as that constructed in \cref{ExistencKHomClasses}.ii) follows via excision.
\end{proof}

We conclude this subsection with a few results on finite-propagation properties and coercivity. They hint towards possible applications to higher indices, secondary invariants, and scalar curvature, though beyond a few remarks we will not pursue this here.

Recall that an operator $T$ on $L^2(\olO)$ is said to have \emph{finite propagation} if there is $R\geq 0$ so that $fTg=0$ whenever $f,g \in C_c(\olO)$ have supports with distance $\geq R$ from another.

\begin{lem}\label{lem_fin_prop}
  Let $D$ be a first-order differential operator with 
  $$
    \sup \left\{\left\|\sigma_D(\xi)\right\|\,\colon\, \xi \in T^*\Omega,\,\|\xi\|\leq 1\right\} < \infty .
  $$
  Let $D_e$ be a local, self-adjoint boundary condition for $D$. Let $\varphi \in C_b(\R)$ have compactly supported Fourier transform. Then, the operator $\varphi(D_e)$ has finite propagation.
\end{lem}

\begin{proof}
  This can be derived from the fact that the wave operator associated to $D_e$ has finite propagation. The latter fact is posed as \cite[Exercise~10.9.13(a)]{HigsonRoe2000} (for compact $\olO$, but this is inessential), and can be proved analogously to the corresponding finite-propagation result for manifolds without boundary, see e.g. \cite[Proposition~8.2.1]{Willett2019}. 
\end{proof}

We denote by $D^*(\olO)$ the $C^*$-algebra generated by all finite-propagation operators $T$ with $[f,T]$ compact for all $f\in C_0(\olO)$, and by $C^*(\olO)$ the ideal in $D^*(\olO)$ generated by all finite-propagation $T$ with $fT$ and $Tf$ compact for all $f\in C_0(\olO)$. It holds that 
$$
  K_*(\olO) \cong K_{*+1}\left( D^*(\olO)/C^*(\olO) \right) ,
$$
so that $K_*(\olO)$ fits into a long exact sequence
$$
  \cdots \longrightarrow K_{*+1}\left(D^*(\olO)\right) \longrightarrow K_*(\olO) \xlongrightarrow{\partial} K_*\left(C^*(\olO)\right) \longrightarrow \cdots ,
$$
see \cite[Section~12.3]{HigsonRoe2000}. While this isomorphism implies that any K-homology class, in particular one defined by a suitable boundary condition, has a representative in $D^*(\olO)$ that an involution modulo $C^*(\olO)$, it is often convenient to know that the cycle itself is already of this form. 

\begin{prop}\label{rep_structure_alg}
  Let $D_e$ be a local, regular, and self-adjoint boundary condition for $D$. Let $\chi$ be a normalizing function. Then, $\chi(D_e)$ is an element of $D^*(\olO)$ that is an involution modulo $C^*(\olO)$.
\end{prop}

\begin{proof}
  Since normalizing functions with compactly supported Fourier transform are dense in the space of all normalizing functions, we may assume $\chi$ to have this property. Then, finite propagation is implied by \cref{lem_fin_prop}, and the necessary compactness statements by \cref{ExistencKHomClasses}.d).
\end{proof}

\begin{rem}
  There is an analogous version of \cref{rep_structure_alg} in the semi-regular case. It is most easily phrased in the setting of \cref{KHomGradedComp}. The boundary condition $D_e$ as defined in that proposition is self-adjoint and local, hence the finite propagation result of \cref{lem_fin_prop} still applies. Letting $\chi$ be any normalizing function and appealing to the compactness statements of \cref{KHomGradedComp}, one still has that $\chi(D_e)$ is in $D^*(\olO)$, but it is now only invertible modulo $D^*(\olO,\dO)$, the ideal of $D^*(\olO)$ generated by all finite-propagation operators $T$ with $Tf$, $fT$ compact for all $f\in C_0(\Omega^\circ)$.
\end{rem}

An important situation in higher index theory on boundaryless manifolds is when the operator in question is coercive, i.e. has a spectral gap around $0$. In this case, one can lift the corresponding K-homology class to a secondary invariant in the K-theory of the $D^*$-algebra. We shall show that the same is possible in the presence of a boundary.

\begin{defn}
  A boundary condition $D_e$ for $D$ is called \emph{coercive} if there is $c>0$ such that 
  \[
    \|D_eu\| \geq c\|u\| \text{ for all } u\in \dom(D_e).\qedhere
  \]
\end{defn}

\begin{ex}[Dirac operator and curvature]\label{ex_dirac_curv}
  Suppose that $\olO$ is spin, and that the metric has uniformly positive scalar curvature $\mathrm{scal}\geq 4c^2 > 0$. The Lichnerowicz formula implies that 
  $$
    \|\Dirac u \|^2 \geq \|\frac{\mathrm{scal}}{4} u\|^2 \geq c^2 \|u\|^2
  $$
  for all $u\in\Gamma_\cc(\Omega,\slashed{S}(\Omega))$. Passing to the graph norm closure yields that $\Dirac_{\min}$ is coercive. If in addition to uniformly positive scalar curvature we assume non-negative mean curvature for the boundary, then it is shown in \cite[Propositions~8.4 \& 8.6]{Baer2025} that both the APS-boundary condition and the chiral boundary conditions are coercive (under some mild boundedness conditions). 
\end{ex}

If $D_e$ is a coercive boundary condition, then we are free to choose a normalizing function $\chi$ taking only the values $\pm 1$ on the spectrum of $D_e$. Then, $\chi(D_e)$ is an involution on the nose. Hence, we conclude the following from \cref{rep_structure_alg}:

\begin{cor}\label{coercive_lift}
  Let $D$ be first-order with bounded principal symbol. Let $D_e$ be a coercive, local, regular, and self-adjoint boundary condition for $D$. Let $\chi$ be a normalizing function with $\chi(\lambda) = \pm 1$ for $\lambda \in \sigma(D_e)$. Then, $\chi(D_e)$ is an involution in $D^*(\olO)$. In particular, $[D_e] \in K_*(D_e)$ admits a canonical lift
  $$
    \rho(D_e) \; \in \; K_*\left( D^*(\olO) \right) .
  $$
\end{cor}

\begin{rem}
  Again, there is a relative version of \cref{coercive_lift}, though it is not as satisfying. It is again phrased easiest in the setting of \cref{KHomGradedComp}. Even if $D_e^+$ is coercive, $D_e^*$ need not be. Thus, $\chi(D_e)$ may have zero as isolated spectral value, hence not be an involution. However, due to the coercivity assumption, it can be shown to be an involution modulo $D^*(\dO\subseteq\olO,\dO)$, the ideal of $D^*(\olO)$ generated by all finite-propagation operators $T$ with $Tf$, $fT$ compact for all $f\in C_0(\olO\setminus\dO)$, and such that there is $R\geq 0$ so that $Tf=fT=0$ if $\supp(f)$ has distance $\geq R$ from $\dO$. We will not give the details here, and restrict ourselves to saying that they are very much analogous to similar statements in the boundaryless case (e.g. \cite[Section~2]{Roe2016}). We also refer to \cite[Section~7.1]{LykoThesis}, where the details can be found for the uniform version of these statements, and to the forthcoming article \cite{LykoAPS}, where the resulting boundary map into $D^*(\dO\subseteq\olO,\dO)$ is connected to secondary invariants on $\dO$.
\end{rem}

\begin{ex}[Chiral boundary conditions and curvature]\label{ex_chiral_lift}
  Let $\olO$ be an odd-dimensional spin manifold. We have noted in \cref{ex_khom_chiral} above that in this case the boundary condition 
  $$
  \hat{\Dirac} \coloneqq \begin{pmatrix}
    0 & \Dirac_- \\ \Dirac_+ & 0
  \end{pmatrix}
  $$ 
  defines an absolute cycle over $\olO$. If $\olO$ has uniformly positive scalar curvature and $\dO$ non-negative mean curvature, then this boundary condition is also coercive by \cref{ex_dirac_curv} (provided that $\dO$ is compact or otherwise sufficiently regular as to fit in the setting of \cite{Baer2025}). Hence, by \cref{coercive_lift} there is a lift
  $$
    \rho(\hat{\Dirac}) \in K_1\left( D^*(\olO) \right) .
  $$
  Conceptually, this class is analogous to the higher $\rho$-invariant for Dirac operators on manifolds of uniformly positive scalar curvature without boundary. However, as noted in \cref{ex_khom_chiral}, for all we know it might be true that both the classes $[\hat{\Dirac}]\in K_0(\olO)$ and $\rho(\hat{\Dirac}) \in K_1( D^*(\olO) )$ are trivial.
\end{ex}

\subsection{Extension to non-local boundary conditions}\label{SecNonlocal}

For the existence of relative K-homology classes in \cref{ExistencKHomClasses} we assumed that the boundary condition in question is local. While \cref{ExAPSChirality} showed that some sort of locality assumption is necessary, our notion of locality is actually stronger than necessary. The next proposition shows that we get relative classes as long as the set of functions preserving $\dom(D_e)$ is sufficiently large as to generate $C_0(\olO)$. This is connected to the construction described in \cite[Section~2.2.2]{Forsyth2019}.

\begin{prop}\label{ClassesNonlocal}
  In the setting of \cref{ExistencKHomClasses}, assume that the boundary condition $D_e$ is semi-regular, and that the set of functions
  $$
    \{f\in C^\infty_c(\olO) \colon f\cdot\dom(D_e)\subseteq\dom(D_e)\}
  $$
  generates the $C^*$-algebra $C_0(\olO)$. Then, 
  $$
    \left( L^2(\olO), \rho, D_e\left( D_e^*D_e + I \right)^{-\sfrac{1}{2}} \right)
  $$
  is a relative Fredholm module over $(\olO,\dO)$. If $D_e$ is also regular and self-adjoint, then it defines an absolute Fredholm module over $\olO$. If $D_e$ is odd/multigraded, then this is again an odd/multigraded Fredholm module.
\end{prop}

\begin{proof}
  The results of \cref{SectionFormalPseudolocal} yield that the commutator of $D_e\left( D_e^*D_e + I \right)^{-\sfrac{1}{2}}$ with every function $f\in C^\infty_c(\olO)$ such that $f\cdot\dom(D_e)\subseteq\dom(D_e)$ is compact. The set of $f\in C_0(\olO)$ that commute with this operator up to compacts is a closed two-sided $^\ast$-ideal. Hence, if the smooth, compactly supported functions preserving $\dom(D_e)$ generate $C_0(\olO)$, then $D_e\left( D_e^*D_e + I \right)^{-\sfrac{1}{2}}$ is pseudo-local over $\olO$. If $D_e$ is also regular and self-adjoint, the remaining properties of an absolute Fredholm module over $\olO$ follow as in the proof of \cref{ExistencKHomClasses}.
\end{proof}

\begin{ex}[Neumann Laplacian]\label{ex_neumann_laplace_trivial}
A function $f\in C^\infty_c(\olO)$ preserves $\dom(\Delta_\mathrm{Neu})$ if and only if $\partial_\nu f = 0$. The set of all such functions is nowhere-vanishing and separates points on $\olO$. Hence by Stone--Weierstrass, $\Delta_\mathrm{Neu}$ satisfies the conditions of \cref{ClassesNonlocal}, thus defines an absolute Fredholm module over $\olO$. For the same reason as for the Dirichlet Laplacian (\cref{DirichletTrivial}), the K-homology class $[\Delta_\mathrm{Neu}]$ is trivial.
\end{ex}

The result can be extended to boundary conditions with even less locality. However, this comes at the expense of changing the total space. Let $D_e$ be a boundary condition for $D$. Let $A$ denote the $C^*$-algebra generated by all $f\in C^\infty_c(\olO)$ with $f\cdot\dom(D_e)\subseteq\dom(D_e)$. Since $A$ is commutative, by Gelfand duality it is isomorphic to $C_0(X_e)$ for some space $X_e$: $A \cong C_0(X_e)$. \cref{MultInt} implies that $A$ contains $C_0(\Omega^\circ)$ as an ideal, and by definition we have $A \subset C_0(\olO)$. Therefore, $X_e$ contains $\Omega^\circ$ as an open subspace, and the inclusion $C_0(X_e) \to C_0(\olO)$ in turn defines a surjection $\olO\to X_e$. Hence, geometrically, $X_e$ arises through quotienting $\dO$ in a way dictated by the boundary condition $D_e$.

\begin{prop}
  In the setting of \cref{ExistencKHomClasses}, assume that the boundary condition $D_e$ is semi-regular, and let $X_e$ be as above. Then, 
  $$
    \left( L^2(\olO), \rho, D_e\left( D_e^*D_e + I \right)^{-\sfrac{1}{2}} \right)
  $$
  is a relative Fredholm module over $(X_e,X_e\setminus\Omega^\circ)$. If $D_e$ is also regular and self-adjoint, then it defines an absolute Fredholm module over $X_e$. If $D_e$ is odd/multigraded, then this is again an odd/multigraded Fredholm module.
\end{prop}

\begin{proof}
  The proof is entirely analogous to those of \cref{ExistencKHomClasses} and \cref{ClassesNonlocal}.
\end{proof}

\begin{ex}[APS-boundary condition]\label{ex_aps_collapse}
  Assume that $\olO$ is spin, and $\dO$ compact. It is shown in \cite[Proposition~4.7]{Battisti2017} that any function $f\in C^\infty_c(\olO)$ that preserves $\dom(\Dirac_\mathrm{APS})$ is locally constant on the boundary. Hence, the APS-boundary condition defines a class in the $K$-homology of the space obtained by collapsing each component of $\dO$ to a point. See also \cite[~Appendix A]{BDT1989}.
\end{ex}

\begin{ex}[Transmission boundary condition]
  Recall the transmission boundary condition from \cref{ex_transmission}. It is regular, self-adjoint, odd in even dimensions, but not local. A smooth function $f\in C^\infty_c(\olO)$ preserves $\dom(\Dirac_T)$ if and only if $f|_{\Sigma_1} = f|_{\Sigma_2}$. It follows that the space $X_T$ can be canonically identified with $M$. Hence, 
  $$
     \left( L^2(\olO), \rho, \Dirac_T\left( \Dirac_T^2 + I \right)^{-\sfrac{1}{2}} \right)
  $$
  is an absolute Fredholm module over $M$. The identification of $L^2(\olO)$ with $L^2(M)$ by gluing together along $\Sigma$ is compatible with the respective $C_0(M)$-representations, and it is not hard to see that $\Dirac_T$ and $\Dirac_M$ agree under this identification. Hence,
  \[
    [\Dirac_T] = [\Dirac_M] \in K_*(M) .\qedhere
  \]
\end{ex}

\section{Further questions}\label{sec_questions}

We discuss a few questions and possible directions for future inquiry. The first points concern concrete examples (or the current shortage thereof), while the later ones target a more general understanding of the absolute classes defined by regular boundary conditions.

\begin{enumerate}
\item\label{bullet_abs_rel} On a closed connected even-dimensional manifold $M$ it holds that
$$
[D^\chi] = \chi(M) \cdot i_*(\mathbf{1}) \in K_0(M),
$$ 
where $i_*(\mathbf{1})$ is the image of a generator of $K_0(\mathrm{pt})$ under the inclusion of a point. Thus the K-homology class of the Euler characteristic operator contains no information beyond its index. On a non-compact manifold without boundary, the class $[D^\chi]$ vanishes. This was proved by Rosenberg \cite{Rosenberg1999}. To what extent is this true for manifolds with boundary, i.e. are $[D^\chi_{\abs}]$ and $[D^\chi_{\rel}]$ in the image of the inclusion of a point, and vanish in the non-compact case?

\item\label{bullet_chiral} In \cref{ex_khom_chiral} we have noted that on odd-dimensional spin manifolds, the chiral boundary conditions define an even-dimensional absolute $K$-homology class 
\begin{equation}\label{offdiag_chiral}
  \left[  \begin{pmatrix}
    0 & \Dirac_-  \\ \Dirac_+ & 0
  \end{pmatrix} \right] \; \in \; K_0(\olO).
\end{equation}
If $\olO$ is compact, the index of this operator is zero \cite[Section~8.5]{BaerBallmann2012}. Hence, it is natural to ask whether the entire $K$-homology class \eqref{offdiag_chiral} is zero. If it is actually non-trivial, this could have interesting consequences for the study of positive-scalar-curvature metrics on $\olO$. Specifically, if $\olO$ has positive scalar curvature and $\dO$ non-negative mean curvature, this boundary condition is coercive, whence its $K$-homology class admits a lift to $K_1(D^*(\olO))$ (see \cref{ex_chiral_lift}). This lift would provide an analogue to the higher $\rho$-invariant in the boundaryless case. Since the higher $\rho$-invariant contains much information about the geometry of positive-scalar-curvature metrics (see e.g. \cite{Piazza2014,Xie2021}), it would be quite desirable to have a counterpart for manifolds with boundary. Hence, we ask: Is the class \eqref{offdiag_chiral} non-trivial? If so, can the associated lift be used to study metrics of positive scalar curvature and non-negative mean curvature on the boundary?

\item The examples of absolute classes we have discussed here range from conjecturally uninteresting beyond their index (absolute/relative boundary conditions, \cref{bullet_abs_rel}) over potentially trivial (chiral boundary conditions, \cref{bullet_chiral}) to known-to-be-trivial (Dirichlet/Neumann Laplacian, \cref{DirichletTrivial} and \cref{ex_neumann_laplace_trivial}). While we do not doubt the existence of non-trivial, interesting examples, we currently do not have any. Hence: What are interesting operators with suitable boundary conditions that gives rise to interesting absolute K-homology classes?

\item\label{bullet_AD} Assume that $D$ is an elliptic operator with $D_d$ and $D_e$ two self-adjoint, local, regular boundary conditions. Consider the exact sequence
\[
K_*(\dO) \to K_*(\olO) \to K_*(\olO,\dO).
\]
Both $[D_d],[D_e] \in K_*(\olO)$ map to the same element in $K_*(\olO,\dO)$ (Items~b) and~c) in \cref{ExistencKHomClasses}) and so $[D_d] - [D_e]$ lifts to a class in $K_*(\dO)$. Is there a concrete way to describe a Fredholm module over $\dO$ that represents this lift?

Agranovich--Dynin established a formula for the difference of Fredholm indices of $[D_d]$ and $[D_e]$, see \cite{AgranovichDynin}. Hence, we are asking for a $K$-homological lift of the Agranovich--Dynin formula.

\item Rempel--Schulze found an Agranovich--Dynin formula in the Boutet de Monvel calculus \cite[Section 3.2.1.3]{RempelSchulze}. They directly provide an elliptic pseudodifferential operator over the boundary of the manifold whose index is the difference of indices of the two boundary conditions. The obvious conjecture here is that the $K$-homological lift sought in \cref{bullet_AD} is exactly given by the class of this operator. Hence: How does one get absolute classes from the Boutet de Monvel calculus?

\item Spaces of (self-adjoint) unbounded operators can be topologized using the \emph{Riesz metric}, which is defined exactly such that the bounded transform $T\mapsto T(T^*T+I)^{-\sfrac{1}{2}}$ is an isometry to the space of bounded operators. The Riesz metric was already investigated in connection to $K$-theory and index theory \cite{Nicolaescu2000,Bandara2019,BoossBavnbeck2005}. For an elliptic operator $D$, we define $\mathrm{BC}(D)$ to be the set of all local, regular, and self-adjoint boundary conditions for $D$. Then, a Riesz-continuous path in $\mathrm{BC}(D)$ gives rise to an operator-homotopy of absolute cycles for $K_1(\olO)$. Hence, there is a well-defined map
$$
  \pi_0\left( \mathrm{BC}(D) \right) \longrightarrow K_1(\olO) .
$$
By viewing $D$ as one of the components of an odd operator, one gets an analogous map from path-components of local, regular boundary conditions for $D$ to $K_0(\olO)$. We note that the image of these maps would be described by the hypothetical $K$-homological Agranovich--Dynin formula of \cref{bullet_AD} above. Therefore: Can one use $K$-homology and a hypothetical $K$-homological Agranovich--Dynin formula to describe spaces of local, regular boundary conditions?

More generally, in \cref{SecNonlocal} we associate to a regular, non-local boundary condition $D_e$ a space $X_e$ so that $D_e$ defines a $K$-homology class over $X_e$. If $D_e$ and $D_d$ are two regular boundary conditions connected by a Riesz-continuous path, how do the spaces $X_e$ and $X_d$ relate? Are they homotopy equivalent to each other? Can these spaces and the resulting $K$-homology classes over them be used to describe spaces of regular, not-necessarily-local boundary conditions?
\end{enumerate}

\printbibliography[heading=bibintoc]

\end{document}